\documentclass[]{amsart}

\usepackage[utf8]{inputenc}
\usepackage[T1]{fontenc}
\usepackage{lmodern}
\usepackage[english]{babel}

\usepackage{graphicx}
\usepackage[numbers]{natbib}

\usepackage{dsfont}
\usepackage{mathtools}
\usepackage{amssymb}
\usepackage{amsthm}
\usepackage{microtype}

\theoremstyle{definition}
\newtheorem{definition}{Definition}
\newtheorem{example}{Example}
\newtheorem{lemma}{Lemma}
\newtheorem{notation}{Notation}
\newtheorem{proposition}{Proposition}

\theoremstyle{plain}
\newtheorem{theorem}{Theorem}

\theoremstyle{remark}
\newtheorem{remark}{Remark}

\usepackage{hyperref}

\begin{document}

\title{Singular Euler--Maclaurin expansion}

\author{Andreas A. Buchheit}
\address{Department of Mathematics, Saarland University, PO 15 11 50,  D-66041 Saarbrücken}
\email{buchheit@num.uni-sb.de}
\author{Torsten Ke{\ss}ler}
\address{Department of Mathematics, Saarland University, PO 15 11 50,  D-66041 Saarbrücken}
\email{kessler@num.uni-sb.de}

\begin{abstract}
  We present the singular Euler--Maclaurin expansion, a new method for the efficient computation of large singular sums that appear in long-range interacting systems in condensed matter and quantum physics. In contrast to the traditional Euler--Maclaurin summation formula, the new method is applicable also to the product of a differentiable function and a singularity. For suitable non-singular functions, we show that the approximation error decays exponentially in the expansion order and polynomially in the characteristic length scale of the non-singular function, where precise error estimates are provided.  The sum is  approximated by an integral plus a differential operator acting on the non-singular function factor only. The singularity furthermore is included in a generalisation of the Bernoulli polynomials that form the coefficients of the differential operator.  We demonstrate the numerical performance of the singular Euler--Maclaurin expansion by applying it to the computation of the full non-linear long-range forces inside a macroscopic one-dimensional crystal with $2\times 10^{10}$ particles. A reference implementation in Mathematica is provided online.
\end{abstract}
\maketitle

\section{Introduction}

Large sums appear everywhere in nature; our macroscopic world is composed  of microscopic particles whose interaction forces determine the properties of the world we live in. Sums with singularities describe discrete long-range interacting systems in condensed matter and quantum physics \cite{campa2014physics}, with examples ranging from the computation of forces and energies in atomic crystals \cite{Dubin1997} to the study of charge transfer in DNA strings \cite{Gaididei1997}. Making predictions about these sums is however in general a difficult task. In many cases, the evaluation of an integral is easier than the evaluation of a sum, either because there are more tools available on the analytical side or because on the numerical side, efficient quadrature schemes are known. The question arises how sums and integrals are related and how we can approximate one by the other.

A part of the answer to this question was given independently by Leonard Euler in 1736 and  by Colin Maclaurin in 1742 \cite{apostol1999euler}. Consider Fig.~\ref{fig:sum_integral}, which provides an illustration for the approximation of a sum (rectangles) by an integral (blue region). In the red parts, the sum dominates the integral, whereas in the green parts, the integral is larger than the sum.  The Euler--Maclaurin (EM)  expansion describes this difference between  sum and integral of a sufficiently differentiable function in terms of derivatives evaluated at the limits of integration plus a remainder integral. 

\begin{figure}
  \centering
  \includegraphics[width=0.85\textwidth]{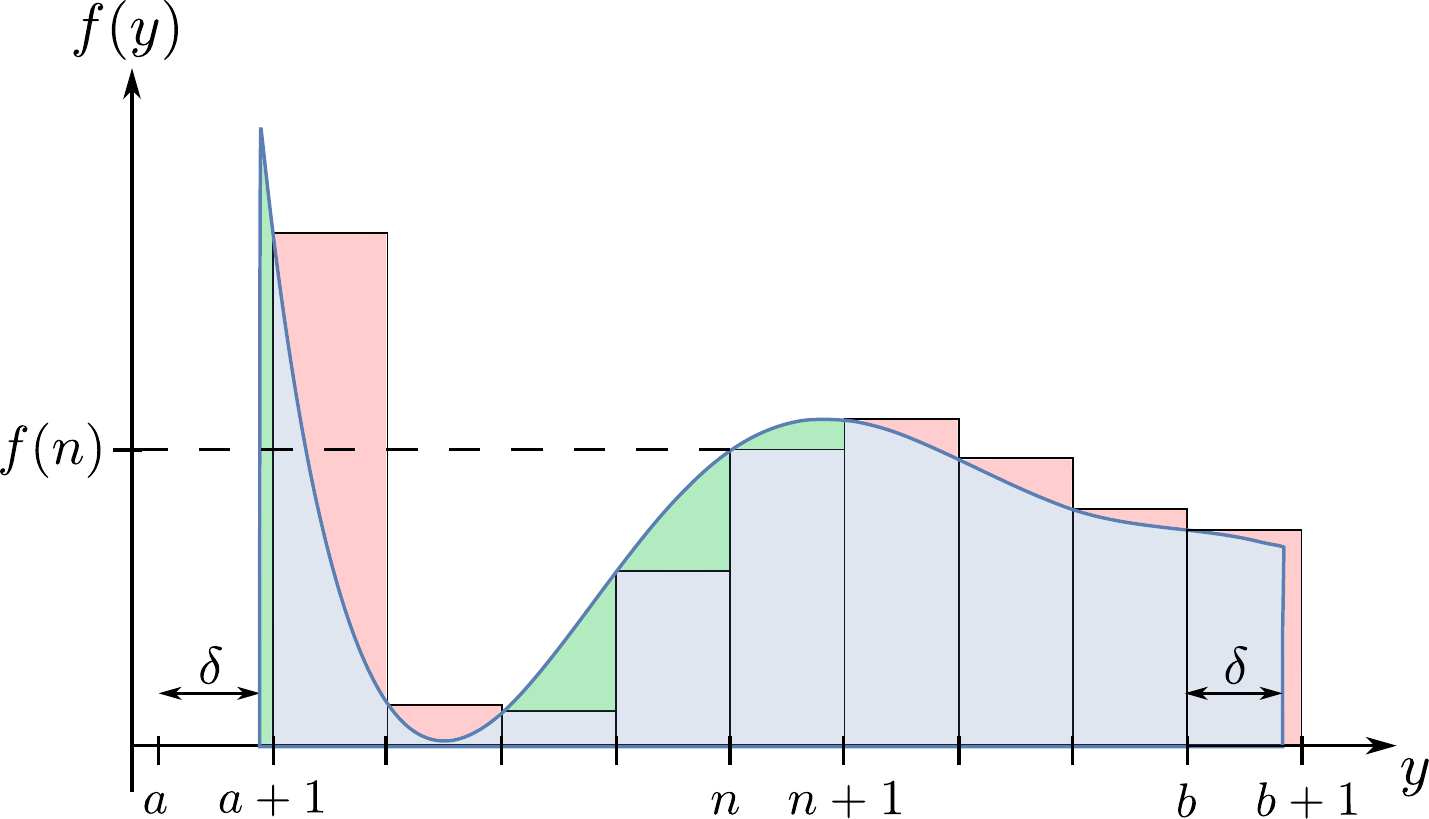}
  \caption{Illustration of the left hand side of the Euler--Maclaurin expansion. Red parts indicate where the integral underestimates the sum, green parts display the opposite case where the integral dominates the sum.}
  \label{fig:sum_integral}
\end{figure}%
Before we state the EM expansion, we introduce the following standard notation: The set of integers is denoted by $\mathds Z$, $\mathds N = \{0,1,2,\dots \}$ are the nonnegative integers, and $\mathds N_+ = \mathds N \setminus \{ 0 \}$ are the positive integers.
Similarly, $\mathds R$ are the real numbers, $\mathds R_+$ denotes the set of positive real numbers and $\mathds R^* = \mathds R \setminus \{ 0 \}$.
Finally, $\mathds C$ are the complex numbers.
In this work, we consider function spaces that are based on differentiable functions.
For an open interval $I \subseteq \mathds R$ and $\ell \in \mathds N$, the vector
space $C^\ell(I)$ consists of functions $f : I \to \mathds C$ being $\ell$ times
differentiable and whose $\ell$-th derivative $f^{(\ell)}$ is continuous.
The derivatives of functions in  $C^\ell(\bar I)$, a subspace of $C^\ell(I)$,
additionally have continuous extensions from $I$ to the closure $\bar I$.
Finally, we define
\[
C^\infty(I) = \bigcap\limits_{\ell = 0}^\infty C^\ell(I),
\]
the space of infinitely differentiable functions on $I$
and analogously $C^\infty(\bar I)$.
The vector space $C^{-1}(I)$ is the space of regulated functions,
\[
C^{-1}(I) = \left\{f: I \to \mathds C: \forall y_0 \in I: \lim\limits_{y \nearrow y_0} f(y), ~\lim\limits_{y \searrow y_0} f(y) \text{ exist}\right\},
\]
where 
\[
\lim\limits_{y \nearrow y_0}, \quad \lim\limits_{y \searrow y_0}
\]
 denote the one-sided limits from left and right.
The vector space $C^{-1}(\bar I)$ is defined analogously, where only the existence of the corresponding one-sided limit is needed at the end points. 
Regulated functions are continuous up to a countable number of points.

For $a,b \in \mathds Z$, $a<b$, $\delta\in(0,1]$, $\ell\in \mathds N$, and for a function $f\in C^{\ell+1}[a+\delta,b+\delta]$, the EM expansion reads
\cite{apostol1999euler,monegato1998euler} 
\begin{align}\label{eq:Euler--Maclaurin-expansion}
  \sum_{n=a+1}^b f(n)=\int\limits_{a+\delta}^{b+\delta} f(y)\,\mathrm d y &- \sum_{k=0}^ {\ell} \frac{(-1)^k}{k!}\frac{B_{k+1}(1+y-\lceil y \rceil)}{k+1} f^{(k)}(y)\bigg\vert^{y=b+\delta}_{y=a+\delta} \notag \\ &+\frac{(-1)^{\ell}}{\ell!} \int \limits_{a+\delta}^{b+\delta}  \frac{B_ {\ell+1}(1+y-\lceil y \rceil)}{\ell+1} f^{(\ell+1)}(y)\,\mathrm d y,
\end{align}
where $\lceil y \rceil$ is the smallest integer larger than or equal to $y$. 
The functions $B_\ell$ are the Bernoulli polynomials, which are uniquely defined by the recurrence relation
\begin{equation}\label{eq:bernoulli_definition}
\begin{aligned}
  B_0(y) =1,  \quad
  B'_\ell(y)=\ell B_ {\ell - 1}(y),\quad 
  \int\limits_0^1 B_\ell(y)\, \mathrm d y=0,\quad \ell\ge 1.
\end{aligned}
\end{equation}
For a derivation of the EM expansion as well as a brief introduction to its history we refer to \cite{apostol1999euler}. 

The applicability of the EM expansion for the approximation of a particular sum is determined by the scaling of the remainder integral on the right hand side of \eqref{eq:Euler--Maclaurin-expansion} with the expansion order $\ell$. 
The remainder integral cannot be evaluated in a numerically feasible way, as the integrand is a piecewise defined function. In the numerical application, an order $\ell$ is chosen, the remainder integral is discarded and the error made by discarding the remainder integral is estimated. If the addend is based on an entire function whose derivatives in addition satisfy certain bounds (more details in the next section), the remainder integral decreases exponentially with $\ell$. This is the ideal case for the EM expansion. For most functions however, even for smooth functions, the error begins to diverge after a certain threshold value of $\ell$ is reached, thus limiting the precision that the EM expansion can offer. 

There have been a number of valuable extensions of the classic works of Euler and Maclaurin, which make the expansion applicable to a larger set of functions. Important progress has been made by Navot in \cite{navot1961extension}, where the EM expansion is generalised to functions with an algebraic singularity at the boundaries of the integration interval. Further generalisations have been developed by Monegato and Lyness, where divergent integrals are regularised by the use of Hadamard finite part integrals \cite{monegato1998euler}. Furthermore, a higher dimensional generalisation of the EM for simple lattice polytopes has been found \cite{karshon2007exact}. We also mention here a recent alternative approach to the EM expansion where the difference between sum and integral is written in terms of integrals only \cite{pinelis2018alternative}.

One particular set of functions, for which the EM expansion fails to converge and which are extremely important in practice are functions that involve an asymptotically smooth singularity, see Definition \ref{def:asymptotically_smooth}. Unfortunately, these functions are of strong practical interest as all physical interactions belong to this kind. In this work, we present the \textbf{singular Euler--Maclaurin expansion (SEM)}, which makes the classic expansion applicable to functions that involve an asymptotically smooth singularity.

This paper targets a wide range of different audiences. Therefore it is organised as follows. Section \ref{sec:main_results} concludes the main results regarding the SEM, offering all the tools needed for application. 
Section~\ref{sec:numerics} covers details to apply the SEM as a numerical tool. We discuss its numerical performance and apply it to a macroscopic long-range interacting crystal as a physically relevant example.
Note that a basic implementation of the SEM in \emph{Mathematica} is provided along with the article\footnote{The code is available online on the github repository \url{https://github.com/andreasbuchheit/singular_euler_maclaurin}}.
Section~\ref{sec:main_derivation} provides the main derivation of the SEM. It includes the proofs of the most important theorems and propositions.
The details of this derivation, including the proofs of rather technical lemmas, are shown in Section~\ref{sec:technical_lemmata}. Finally, in Section  \ref{sec:outlook}, we make our concluding remarks.
\section{Main result and notation}\label{sec:main_results}

In this chapter, we formulate the SEM expansion for intervals $[a+\delta,b+\delta]$, with $a,b \in \mathds Z$, $\delta \in (0,1]$, which is in particular applicable, if the function
\[
f: [a+\delta, b+\delta] \to \mathds C,
\]
splits into two factors
\begin{equation}
  f(y)=s(y-x)g(y), \quad y \in [a+\delta, b+\delta],
  \label{eq:factorisation_of_f}
\end{equation}
where $x \in \mathds Z$, $s\in C^\infty(\mathds R^*)$ has a singularity at $0$ and
$g: [a+\delta, b+\delta] \to \mathds C$ is a sufficiently differentiable function.
We apply the following general strategy:  singularities at $x$ or other smooth functions whose derivatives increase quickly with the derivative order are included in $s$. The function $s$ limits the applicability of the standard EM expansion to $f$ and therefore requires a special treatment. In practice, the function $s$ often represents a pairwise long-ranged interaction potential or the one-dimensional forces generated by such a potential. We refer to $s$ in the following as the interaction. The remaining factor $g$ includes well-behaved functions, whose derivatives increase sufficiently slowly with the derivative order. The slower the derivatives increase, the better are the convergence rates. 

We briefly outline our presentation of the SEM expansion. We first discuss the properties of the function $s$ that are required for the expansion. The interaction is subsequently made integrable by introducing an exponential regularisation. We then make use of the integrability of the regularised interaction and define the Bernoulli--$\mathcal A$ functions, a generalisation of the periodised Bernoulli polynomials in which we encode $s$. These functions then form the coefficients of the differential operator of the SEM, replacing the Bernoulli polynomials in the standard EM expansion \eqref{eq:Euler--Maclaurin-expansion}. The differential operator acts on $g$ only, avoiding the fast increase in the derivatives of $s$ that causes the breakdown of the standard EM expansion. A finite-order approximation of this differential operator leads to the finite-order SEM expansion. For smooth $g$, whose derivatives fulfil certain bounds, we take the order of the expansion to infinity, leading to the infinite order SEM. 

Before moving on to the formulation of the SEM expansion, we need to specify the admissible set of functions for the interaction: the function $s$ has to be asymptotically smooth~\cite[Sec. 3.2]{bebendorf2008hierarchical}.
\begin{definition}[Asymptotically smooth functions]
\label{def:asymptotically_smooth}
A function $s \in C^\infty (\mathds R^*)$ is called asymptotically smooth if there exist $c>0$  and $\gamma\ge 1$ such that 
  \begin{equation}
    \left|s^{(\ell)}(y)\right|\leq c \, \ell! \, \gamma^\ell \, |y|^{-\ell} \, |s(y)|,
    \label{eq:function_class}
  \end{equation}
  for all $y \in \mathds R^*$ and $\ell\in \mathds N$. We denote the vector space of all asymptotically smooth functions by $S$.  
\end{definition}

This set includes entire functions like polynomials, but also a broad set of functions with singularities.
Typical examples for asymptotically smooth functions are
\begin{equation}
  s(y)=|y|^{-\nu}, \quad y \in \mathds R^*,
  \label{eq:s_algebraic}
\end{equation}
for $\nu \in \mathds R$ which are singular for $\nu>0$.

\begin{remark}\label{rem:gamma-algebraic-singularity}
  For $s$ in \eqref{eq:s_algebraic}, the constant $\gamma$ in \eqref{eq:function_class} equals $1$ in the case $\nu \le 1$.
  For $\nu > 1$, then $\gamma=1+\varepsilon$ for an arbitrary $\varepsilon>0$, see the proof at the beginning Section~\ref{sec:technical_lemmata}.
\end{remark}

We can further classify asymptotically smooth functions by their growth rate at infinity. 
\begin{definition}
  \label{def:s_alpha}
We define $S_\alpha$, $\alpha\in \mathds R$, as the vector space of all $s\in S$ for which there exists $c_0>0$ such that
\begin{equation}
  |s(y)|\le c_0 |y|^\alpha,\quad |y|>1.
  \label{eq:alpha_relation}
\end{equation}
\end{definition}

\begin{remark}\label{rem:s_alpha-gronwall}
From Grönwall's lemma follows immediately that
\begin{equation}
  S=\bigcup_{\alpha\in \mathds R}S_\alpha.
\end{equation}
See Section~\ref{sec:technical_lemmata} for a proof.
\end{remark}

From Definition \ref{def:asymptotically_smooth}, we find that the $\ell$th derivative of the interaction may scale with the factorial of $\ell$. It is this fast increase in the derivatives that causes the breakdown of the EM expansion, and therefore, taking derivatives of $s$ has to be avoided. It turns out that we can integrate $s$ instead. However,  $s$ is in general not integrable on $[1,\infty)$; take for instance $\nu=1$ in \eqref{eq:s_algebraic}. 
This challenge is overcome by using an exponentially decaying regularisation of the interaction.
\begin{notation}
Let $s\in S$. The exponentially weighted interaction reads
\begin{equation}
s_\beta(y)=s(y)e^{-\beta |y|}, \quad y \in \mathds R^*,
\label{eq:sbeta}
\end{equation}
with $\beta \geq 0$.
\end{notation}
For $\beta \searrow 0$, the weighting is gradually removed and the interaction regains its original range.
It is crucial here that the reduction of the interaction range is introduced not as a sharp cut-off, but in a smooth way, such that $s_\beta$ remains asymptotically smooth. 

The basic object from which the SEM expansion is deduced is the function $\mathcal C $. 
\begin{definition}
  Let $s\in S$. We define $\mathcal C :\mathds R_+ \times \mathds R_+\to \mathds  C$ as 
  \begin{equation}
    \mathcal C (y,\beta)=\sum_{n=\lceil y \rceil}^\infty s_\beta(n)-\int\limits_{y}^\infty s_\beta(z)\, \mathrm d z,
    \label{eq:dbeta}
  \end{equation}
  \label{def:c_function}
\end{definition}
In the following, calligraphic symbols indicate an explicit dependence on the interaction $s$.
The function $\mathcal C $ quantifies the difference between sum and integral of the regularised interaction. It can be used to generate a replacement for the periodised Bernoulli polynomials in \eqref{eq:Euler--Maclaurin-expansion}, in which we encode all information about the interaction $s$.
\begin{definition}[Bernoulli--$\mathcal A$ functions]
  \label{def:bernoulliA}
  Let $s \in S$.
  The Bernoulli--$\mathcal A$ functions,
  \[
  \mathcal A_\ell : \mathds R_+ \to \mathds C, \quad \ell \in \mathds N,
  \]
  are defined as the coefficients in the power series
  \[
  e^{\beta \xi} \mathcal C (\xi, \beta) =
  \sum\limits_{\ell = 0}^\infty \mathcal A_\ell (\xi) \frac{\beta^\ell}{\ell!}, \quad \xi > 0.
  \]
  We say that $(\mathcal A_\ell (\xi))_{\ell \in \mathds N}$ is exponentially generated by
  \[
  \mathcal G_\xi (\beta) = e^{\beta \xi} \mathcal C (\xi, \beta)
  \]
  and refer to $\mathcal G_\xi $ as the generating function.
\end{definition}
\begin{figure}
  \centering 
  \includegraphics[width=0.7\textwidth]{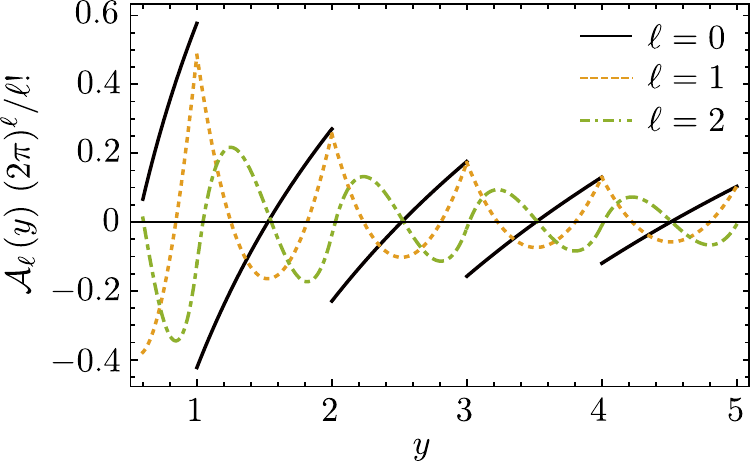}
  \caption{Generalised Bernoulli functions $\mathcal A_\ell $ for $s(y)=|y|^{-1}$.}
  \label{fig:bernoulliA}
\end{figure}
The Bernoulli--$\mathcal A$ functions replace the periodic extension of the Bernoulli polynomials in the differential operator part and the remainder of the EM expansion. We display them for $s(y)=|y|^{-1}$ in Fig.~\ref{fig:bernoulliA}. 
\begin{remark}
  For $s=1$, we recover the periodised Bernoulli polynomials,
  \begin{equation*}
    \mathcal A_\ell(y) = \frac{B_{\ell +1}(1+y-\lceil y\rceil)}{\ell+1}, \quad \ell \in \mathds N,~y > 0.
  \end{equation*}
  \end{remark}

\begin{proof}
For $y > 0$ and $\beta > 0$,
\[
\begin{aligned}
e^{\beta y} \mathcal C(y, \beta) &= e^{\beta y}\sum\limits_{n = \lceil y \rceil}^\infty e^{-\beta n} - e^{\beta y} \int\limits_{y}^\infty e^{-\beta z} \, \mathrm d z
=\frac{e^{\beta(1 + y - \lceil y \rceil)}}{e^{\beta} - 1} -\frac{1}{\beta} \\
&= \frac{1}{\beta}\left( \frac{\beta e^{\beta(1 + y - \lceil y \rceil)}}{e^{\beta} - 1} - 1 \right).
\end{aligned}
\]
The first term in brackets is the exponential generating function for the Bernoulli polynomials evaluated at $1 + y - \lceil y \rceil$ \cite[Sec. 1.13, Eq. (2)]{Batemann1953a}.
Since $B_0 = 1$, we have
\[
e^{\beta y} \mathcal C(y, \beta) = \sum\limits_{\ell=1}^\infty B_{\ell}(1 + y - \lceil y \rceil) \frac{\beta^{\ell - 1}}{\ell !}
=
\sum\limits_{\ell = 0}^\infty \frac{B_{\ell + 1}(1 + y - \lceil y \rceil)}{\ell + 1} \frac{\beta^\ell}{\ell !}.
\] 
\end{proof}

We now define the SEM differential operator as follows:
\begin{definition}[SEM operator]\label{def:sem_operator}
  For $s\in S$ and $\xi \in \mathds R_+$, we define the differential operator of infinite order
  \begin{equation}
    {\mathcal D}_\xi = \sum_{\ell=0}^\infty \frac{1}{\ell!} \mathcal A_\ell (\xi)\, (-D)^\ell,
    \label{eq:Ddeltaxi}
  \end{equation}
  with $D$ is the derivative operator. We call $\mathcal D_\xi $ the SEM operator. It formally reads 
  \begin{equation*}
      {\mathcal D}_\xi  = \mathcal G_\xi(-D),
  \end{equation*}
  and for $\ell \in \mathds N$ the finite order approximations ${\mathcal D}_\xi^{(\ell)}$ are given by 
  \begin{equation*}
    {\mathcal D}_\xi^{(\ell)} =\sum_{k=0}^\ell \frac{1}{k!} \mathcal A_k (\xi)\, (-D)^k.
  \end{equation*}
\end{definition}

We present the first theorem of this work, the finite order SEM expansion.
\begin{theorem}[Finite order SEM] \label{th:singular_euler_finite}
  For $x,a,b\in \mathds Z$, with $x\le a< b$, and $\delta \in(0,1]$, let $f$ factor into
  \begin{equation*}
    f(y)=s(y-x)g(y),
  \end{equation*}
  where $s\in S$ and $g\in C^{\ell+1}[a+\delta,b+\delta]$, $\ell \in \mathds N$. Then,
  \begin{equation*}    
  \begin{aligned}
    \sum_{n=a+1}^b f(n)&=\int \limits_{a+\delta}^{b+\delta} f(y)\, \mathrm d y
    - \Big({\mathcal D}^{(\ell)}_{y-x}  g\Big) (y)\bigg\vert_{y=a+\delta}^{y=b+\delta}
    +\frac{(-1)^\ell}{\ell!}\int \limits_{a+\delta}^{b+\delta} \mathcal A_ {\ell} (y-x) g^{(\ell+1)}(y) \,\mathrm d y.
    \label{eq:new_great_euler_finite}
  \end{aligned}
\end{equation*}
\end{theorem}
The SEM operator only acts on $g$, not on $f$, and therefore differentiation of $s$ is avoided.

In order to perform the limit $\ell\to\infty$ in  Theorem~\ref{th:singular_euler_finite}, the function $g$ has to belong to the set of functions of exponential type. This sets a growth condition on its derivatives.
\begin{definition}
Let $g$ be entire. If there exists $\sigma >0$ such that for every $\epsilon>0$, there is $M_{\epsilon}>0$ with
\begin{equation*}
  \left|g^{(\ell)}(y)\right|\le M_{\epsilon}(\sigma+\epsilon)^{\ell}e^{(\sigma+\epsilon)|y|},\quad y \in \mathds R,~\ell\in \mathds N,
\end{equation*}
we say that $g$ is of \textbf{exponential type ${\sigma}$}. By $E_\sigma$ we denote the vector space of all functions of exponential type $\sigma$.
\end{definition}

The classical paper~\cite{carmichael1934functions} gives an exhaustive review of functions of exponential type.
The admissible range for $\sigma$ depends on $s$,
in particular on $\gamma$ from Definition~\ref{def:asymptotically_smooth}.

\begin{theorem}[Infinite order SEM]
  \label{th:singular_euler_entire}
  For $x,a,b\in \mathds Z$, with $x\le a< b$, and $\delta \in(0,1]$, let $f$ factor into
  \begin{equation*}
    f(y)=s(y-x)g(y),
  \end{equation*}
  where $s\in S$ and $g\in E_\sigma$ with $\sigma<2\pi / (1+\gamma)$. Then,
  \begin{equation}
    \sum_{n=a+1}^b f(n) = \int\limits_{a+\delta}^{b+\delta} f(y) \, \mathrm d y - \Big({\mathcal D}_{y-x}   g\Big)(y)\bigg\vert_{y=a+\delta}^{y=b+\delta}.
    \label{eq:diff_op}
  \end{equation}
\end{theorem}

\begin{remark}
For simplicity, we have formulated the Theorems \ref{th:singular_euler_finite} and \ref{th:singular_euler_entire} such that $x$ is positioned to the left of the interval $[a+1,b]$. The Theorems can however also be applied in case that $x$ is positioned to the right of this interval by simply reflecting it about $x$. Consider $a<b<x$. Then the reflected interval is $[\tilde a+1,\tilde b]$, where 
\begin{equation*}
  \tilde a= 2x-(b+1),\quad \tilde b= 2x-(a+1),
\end{equation*}
and thus $x\le \tilde a < \tilde b$.
Using the reflected interval, we can transform the sum as follows
\begin{equation*}
  \sum_{n=a+1}^b s(n-x)g(n)=\sum_{n=\tilde a+1}^{\tilde b}\tilde s(n-x)\tilde g(n),
  \label{eq:interval_transformation}
\end{equation*}
with the functions $\tilde s \in S$ and $\tilde g$ such that 
\begin{equation*}
  \tilde s(y)=s(-y),\quad \tilde g(y)=g(2x-y).
\end{equation*}
Thus the SEM becomes applicable to the right hand side of \eqref{eq:interval_transformation}.
\end{remark}

\begin{remark}
For the prototypical asymptotically smooth interactions
\[
s(y) = |y|^{-\nu}, \quad y \in \mathds R^*,
\]
with $\nu \in \mathds R$, the upper bound on $\sigma$ can be improved to $2 \pi$.
This is a consequence of the next example:
\end{remark}

\begin{example}\label{ex:representation-zeta-function}
  For $s(y)=|y|^{-\nu}$, $y \neq 0$, with $\nu \in \mathds R$, the radius of convergence for the generating function $\mathcal G_y$
  is $2 \pi$. The functions $(\mathcal A_\ell)_{\ell \in \mathds N}$ are given by 
  \begin{equation}
    \mathcal A_\ell (y)=\sum_{k=0}^\ell (-1)^k \binom{\ell}{k} y^{\ell-k} \bigg(\zeta(\nu-k,\lceil  y \rceil)-\frac{y^{-(\nu-k-1)}}{\nu-k-1}\bigg),
    \quad \ell \in \mathds N,
  \end{equation}
  with $\zeta(\cdot,\cdot)$ the Hurwitz zeta function,
  \begin{equation}
    \zeta(z,q)=\sum_{n=0}^\infty \frac{1}{(n +q)^z},\quad z>1,~q>0,
  \end{equation}
  and analytically continued to the complex plane for $z \neq 1$ \cite[Sec. 1.10]{Batemann1953a}.
  For an integral $\nu$, the coefficients are well-defined in the limit,
  \[
  \lim\limits_{\nu \to k+1} \left( \zeta(\nu - k, \lceil y \rceil) - \frac{y^{-(\nu - k - 1)}}{\nu - k -1} \right)
  = \gamma_e - H_{\lceil y \rceil - 1} - \log y, \quad k \in \mathds N,
  \]
  where $\gamma_e$ is the Euler--Mascheroni constant and $H_k$ denotes the $k$th harmonic number,
  \[
  H_k = \sum_{j=1}^k \frac{1}{j}.
  \]
  \end{example}
\section{Numerical Application} \label{sec:numerics}
We demonstrate the numerical performance of the SEM expansion by applying it to the calculation of long-range forces in a macroscopic one-dimensional crystal lattice. The SEM expansion naturally provides the answer to the question how to correctly include the discreteness of a 1D crystal within a continuum formulation that avoids discrete lattice sums and is therefore numerically feasible for all asymptotically smooth interaction forces. We consider a  particularly difficult scenario, the case where the interaction potential is the 3D Coulomb repulsion, which decays algebraically with an exponent equal to the system dimension. Then the discreteness of the crystal has an observable effect on the forces at all scales, which makes a continuum approximation challenging \cite{Dubin1997}. 

We consider a one-dimensional crystal of $2N+1$ particles, $N\in \mathds N$, and denote the particle positions as $x_j\in \mathds R,~j=-N,\dots,N$. The particles are displaced from an equidistant grid with lattice constant $h>0$,
\begin{equation}
  x_j=j h+u(jh),\quad j\in{-N,\dots,N},
\end{equation} 
through a smooth displacement function $u$.

\begin{figure}
  \centering 
  \includegraphics[width=0.9\textwidth]{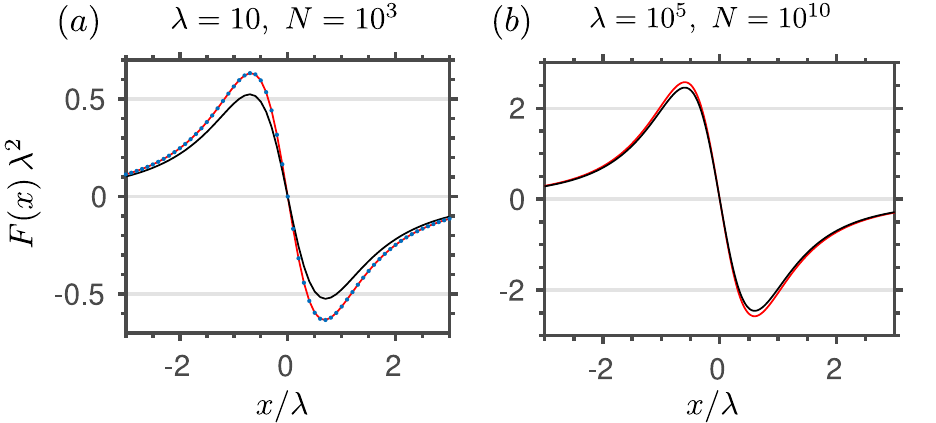}
  \caption{Forces $F$ as a function of distance $x$ in centre of a kink for different choices of the kink width $\lambda$ and the particle number $N$. The red line shows the first order approximation of the singular Euler--Maclaurin expansion, the blue dots display the exact forces. In panel (a), the approximation error in the maximum norm over the whole chain is smaller than $3\times 10^{-7}$ and the relative error is smaller than $8\times 10^{-5}$. The black line shows the approximation of the discrete sum by an integral only.  }
  \label{fig:macro}
\end{figure}
If the interaction energy $V\in S$ between two particles decays algebraically with their distance $x$,
\begin{equation}
  V(x)=c_\nu |x|^{-\nu}, \quad \nu>0,
\end{equation}
the force acting on the particle with reference position $x$ reads
\begin{equation}
  F(x)=-\sum_{\substack{n=-N\\ n \neq x/h}}^{N} V'\Big((x-h n)+u(x)-u(hn)\Big), \quad x\in h\{-N,\dots, N \}.
\end{equation}
All physical dimensions are from now on removed, where we write positions in units of $h$ and forces in units of $V''(h)h$. Then the forces follow as 
\begin{equation}
  F(x)=\sum_{n=-N}^{x-1} f(n)+ \sum_{n=x+1}^N f(n),\quad x\in\{-N,\dots,N\},
\end{equation}
where the function $f$ factors into
\begin{equation}
  f(y)=s(y-x)g(y)
\end{equation}
with $s\in S$ and $g$ smooth such that
\begin{equation}
  s(y)=\mathrm{sgn}(y)~|y|^{-(\nu+1)},\quad g(y)=-\frac{1}{\nu + 1}\bigg(1+\frac{u(y)-u(x)}{y-x}\bigg)^{-(\nu+1)}.
  \label{eq:force_sum}
\end{equation}

\begin{figure}
  \centering 
  \includegraphics[width=0.85\textwidth]{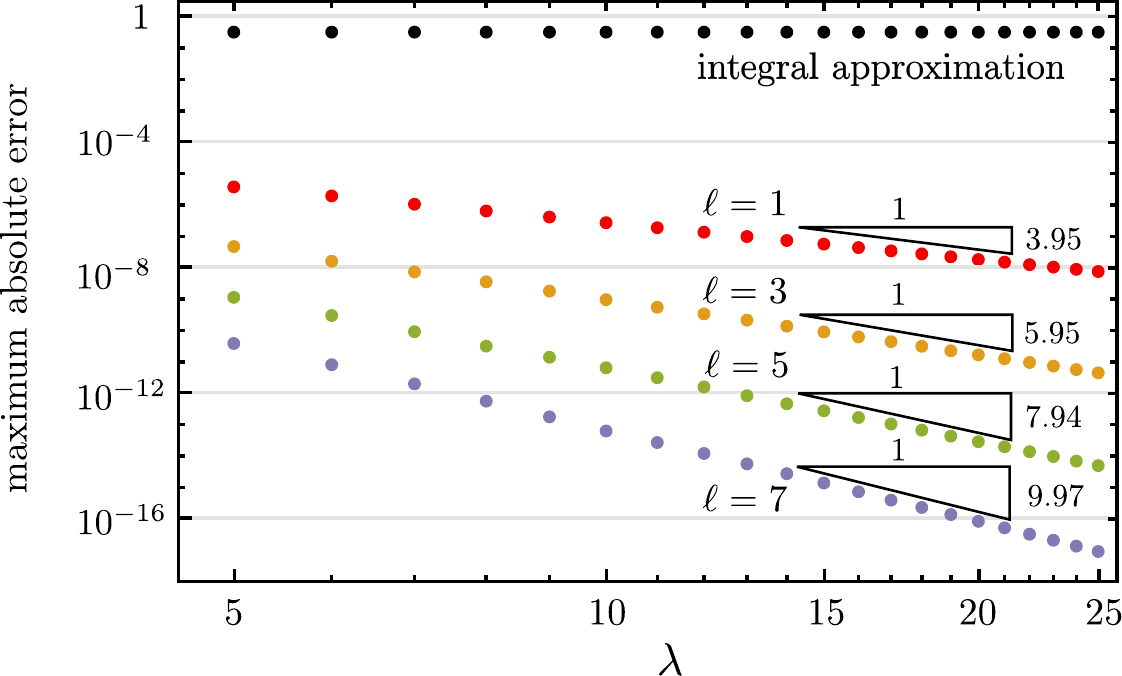}
  \caption{Maximum absolute error for $N=200$ as a function of $\lambda$ for different orders $\ell$ of the singular Euler--Maclaurin expansion.}
  \label{fig:error_scaling}
\end{figure}

For our numerical study, we make the parameter choice $\nu=1$, corresponding to the 3D Coulomb interaction restricted to 1D.  The displacement function is chosen as the integral of a normalised Lorentzian
\begin{equation}
  u(y)=\int \limits_{-\infty}^{y/\lambda} \frac{1}{\pi}\frac{1}{1+z^2}\, \mathrm d z,
\end{equation}
which describes the simplified profile of a kink, an extended defect in the crystal, where
\begin{equation}
  \lim_{y\to \infty} u(y)=1,\qquad \lim_{y\to -\infty} u(y)=0.
\end{equation}
Kinks typically arise when an additional nonlinear potential is applied to the crystal. The parameter $\lambda>0$ controls the width of the kink.
For further details regarding kinks in condensed matter physics, see \cite{Manton2004Topological}.

We compute the forces in \eqref{eq:force_sum} by using the SEM expansion in Theorem \ref{th:singular_euler_finite} up to order $\ell$, choosing $\delta=1$.  The differential operator is first evaluated symbolically and is then subsequently applied to the function $g$. The runtime for the adaptive numerical integration and for the computation of the derivatives is then essentially independent of $N$ for a single force evaluation. The SEM is implemented in \emph{Mathematica} and a working example is provided along with this article. 

We apply the SEM expansion to the computation of the full nonlinear long-range Coulomb forces in a crystal of macroscopic size. In Fig.~\ref{fig:macro} we first display the forces in the centre of the kink in case of a microscopic crystal with $\lambda=10$ and $N=1000$ in (a) and then for a macroscopic crystal with $\lambda=10^5$ and $N=10^{10}$ in (b). For a typical lattice constant $h\approx 10^{-10}\,\mathrm{m}$, the crystal in (b) exhibits a total length of two metres. The first order SEM (red line) is compared to the approximation of the sum by an integral only (black line). The blue dots in (a) show the exact discrete particle forces. The calculation of the exact forces in (b) is not numerically possible anymore due to the large number of particles. We find that the SEM reproduces the exact forces in (a) very precisely both at the chain edges as well as inside the kink in the centre. The absolute error is less than $3\times 10^{-7}$ for all particles with $4$ digits of precision. The integral approximation however shows a significant error. 
In (b) both the particle number and the size of the kink are increased to macroscopic scales. Even at the macro scale, the SEM shows a visible and important difference to the integral approximation. The maximum forces at the kink centre scale as $\log(\lambda)/\lambda^2$ while the first order SEM contribution in the centre scales as $\lambda^{-2}$. Both are independent of $N$ as $N\to \infty$. As the logarithm increases too slowly with $\lambda$ in order to dominate the first order SEM contribution, the discrete particle nature of the crystal remains very much relevant even in the thermodynamic limit. We have therefore shown that the often made claim, that a lattice sum may be replaced by an integral if the underlying charge distribution is sufficiently broad (see e.g. \cite[Eq.~(5.3)]{Braun1990Kinks}), is incorrect in case of $\nu=1$ in one dimension.

We now analyse the scaling of the absolute error in the maximum norm for different SEM orders $\ell$ and kink widths $\lambda$ for $N=200$. The results are displayed in Fig.~\ref{fig:error_scaling}. The smaller particle number is chosen, such that the exact forces can still be computed efficiently for all particles. We find that in case of the integral approximation (black dots), the maximum absolute error does not scale with $\lambda$. The maximum error occurs at the chain edges, the error in the centre scales as $\lambda^{-2}$. Note that the inclusion of the zero order SEM contribution already compensates the error at the edges, with the maximum error now appearing close to the kink centre, in  the region where the derivatives of $u$ are large. For $\ell=1$ (red dots), the first order SEM, the error scales approximately as $\lambda^{-4}$. For odd orders $\ell$ we find that the error scaling coefficient is approximately $\ell+3$. The exact scaling coefficients calculated from a linear fit of the last 5 data point is given in Fig.~\ref{fig:error_scaling}. For $\ell = 7$ (purple dots) and $\lambda=25$, the SEM offers an absolute error smaller than $10^{-17}$ which corresponds to at least $13$ digits of precision for all forces.

The analysis of the error scaling shows that an inclusion of the SEM correction is important for the correct prediction of long-range forces, even more so if finite chains with edges are considered, where the integral approximation suffers a complete break down independent of the choice of $\lambda$. A very regular scaling of the error with $\lambda$ is observed, with the first order SEM already offering a $\lambda^{-4}$ scaling.

The SEM expansion is applicable to all asymptotically smooth interactions. In particular, this includes all standard interaction forces and energies that appear in nature.
\section{Derivation of the singular Euler--Maclaurin expansion}\label{sec:main_derivation}

In this section, we lay out the proofs of the main results whilst skipping technical details.
To this end, we moved the proofs of most lemmas in a separate section.
Before we present the basis for the theorems, namely  the zero order SEM expansion, we collect several properties of the function $\mathcal C$ and its derivatives.
\begin{lemma}\label{lem:Cs_beta_properties}
Let $s \in S$.
For $y > 0$, the function
\[
(0, \infty) \to \mathds C,~\beta \to \mathcal C (y, \beta)
\]
is infinitely differentiable,
\[
\partial_\beta^\ell \mathcal C (y, \beta) = (-1)^\ell
\left( 
\sum\limits_{n = \lceil y \rceil}^\infty n^\ell s(n) e^{-\beta n} - \int\limits_y^\infty z^\ell s(z) e^{-\beta z} \, \mathrm d z,
\right), \quad \ell \in \mathds N,~\beta > 0.
\]
and all derivatives decay exponentially for $|y| \to \infty$.
Furthermore, all derivatives have a continuous extension for $\beta = 0$, i.e. the limit
\[
\lim\limits_{\beta \searrow 0} \partial_\beta^\ell \mathcal C (y, \beta)
\]
exists for all $\ell \in \mathds N$.
\end{lemma}

\begin{lemma}\label{lem:Cs_y_properties}
Let $s \in S$ and $\ell \in \mathds N$, $\beta \geq 0$. The function 
\[
(0, \infty) \to \mathds C,~ y \mapsto \partial_\beta^\ell\mathcal C (y, \beta)
\]
is infinitely differentiable on $\mathds R_+ \setminus \mathds N$ and obeys the jump relation
\[
\lim\limits_{y \nearrow n} \partial_\beta^\ell \mathcal C (y, \beta) - \lim\limits_{y \searrow n} \partial_\beta^\ell \mathcal C (y, \beta)
=
(-1)^\ell n^\ell s_\beta(n)
\]
for all $n \in \mathds N \setminus \{ 0 \}$.
\end{lemma}

With above lemmas, especially the jump relations, we can formulate the zero order SEM expansion.

\begin{proposition}
  \label{prop:first_integral}
  Let $x,a,b\in \mathds Z$ with $x\le a<b$ and $\delta\in (0,1]$. Let $f$ factor into
  \begin{equation*}
    f(y)=s(y-x) g(y),
  \end{equation*}
  with $s \in S$ and $g \in  C^{1}[a+\delta,b+\delta]$.
  Then,
    \begin{align}
      &\sum_{n=a+1}^b f(n)- \int \limits_{a+\delta}^{b+\delta} f(y)\, \mathrm d y \notag \\  &= -\lim_{\beta \searrow 0} \mathcal C (y-x,\beta)g(y)\Big\vert_{y=a+\delta}^{y=b+\delta}+\lim_{\beta \searrow 0}\int \limits_{a+\delta}^{b+\delta} \mathcal C (y-x, \beta) g'(y) \,\mathrm d y.
      \label{eq:start_of_proof1}
    \end{align} 
\end{proposition}
\begin{proof}
  First note that by Lemma~\ref{lem:Cs_y_properties}, the  function $\mathcal C (\cdot, \beta)$ exhibits discontinuities at $n\in \mathds N_+$ but is smooth for $y\in \mathds R_+\setminus \mathds N$ with the properties
\begin{equation}
  \lim_{\epsilon\searrow 0} \Big( \mathcal C (n-\epsilon,\beta) -\mathcal C (n+\epsilon,\beta) \Big) = s_\beta(n),\quad n\in \mathds N_+,
  \label{eq:c0_property_jump}
\end{equation}
and  
\begin{equation}
  \partial_y \mathcal C (y,\beta)=s_\beta(y),\quad \text{for}~y\in \mathds R_+\setminus \mathds N.
  \label{eq:d0_property_derivative}
\end{equation}
By property \eqref{eq:c0_property_jump} the sum on the left hand side of \eqref{eq:start_of_proof1} reads
\begin{align}
&\sum_{n=a+1}^b f(n)\notag \\ &=\lim_{\beta \searrow 0}\sum_{n=a+1}^b \lim_{\epsilon \searrow 0} \Big(\mathcal C (n-x-\epsilon,\beta)-\mathcal C (n-x+\epsilon,\beta)\Big) g(n),
\label{eq:sum_fx}
\end{align}
where the weighting of $s$ is removed by the limit $\beta \searrow 0$.
Subsequently  \eqref{eq:sum_fx} is divided into two separate sums. An index shift is performed in the sum that includes the terms $\mathcal C (n-x-\epsilon,\beta)$ resulting in the expression
\begin{align*}
  &\sum_{n=a+1}^b f(n)\\&=\lim_{\beta \searrow 0} \Big(\lim_{\epsilon \searrow 0}\sum_{n=a}^{b-1}  \mathcal C (n+1-x-\epsilon,\beta)g(n+1)-\lim_{\epsilon\searrow 0}\sum_{n=a+1}^{b} \mathcal C (n-x+\epsilon,\beta) g(n) \Big).
\end{align*}
Recombination of the two sums yields
\begin{align}
  &\sum_{n=a+1}^b f(n) \notag \\&=\lim_{\beta \searrow 0}\lim_{\epsilon \searrow 0}   \Bigg(\sum_{n=a+1}^{b} \mathcal C (y-x,\beta) g(y) \Big\vert_{y=n+\epsilon}^{y=n+1-\epsilon} -\mathcal C (y-x,\beta) g(y)\bigg\vert_{y=a+1-\epsilon}^{y=b+1-\epsilon}\Bigg),
  \label{eq:sum_proof1}
\end{align}
where the second term results from an adjustment of the differing summation intervals.
Going back to the integral on the left hand side of \eqref{eq:start_of_proof1}, we use property \eqref{eq:d0_property_derivative} and rewrite it as
\begin{align*}
 &\int \limits_{a+\delta}^{b+\delta} f(y)\, \mathrm d y= \lim_{\beta \searrow 0}\lim_{\epsilon \searrow 0}\bigg(\sum_{n=a+1}^b  \int\limits_{n+\epsilon}^{n+1-\epsilon} \partial_y \mathcal C (y-x,\beta)\, g(y) \,\mathrm d y \notag \\
&+\int\limits_{a+\delta-\epsilon}^{a+1-\epsilon} \partial_y \mathcal C (y-x,\beta) g(y) \,\mathrm d y-\int\limits_{b+\delta-\epsilon}^{b+1-\epsilon} \partial_y \mathcal C (y-x,\beta) g(y)\, \mathrm d y\bigg).
\end{align*} 
Integration by parts on all integrals in \eqref{eq:int_proof1} in order to remove the derivatives of $\mathcal C (\cdot,\beta)$ from the expression yields
\begin{align}
  &\int \limits_{a+\delta}^{b+\delta} f(y)\, \mathrm d y \notag \\ &=\lim_{\beta \searrow 0}\lim_{\epsilon \searrow 0} \Bigg(   \sum_{n=a+1}^{b} \mathcal C (y-x,\beta) g(y) \Big\vert_{y=n+\epsilon}^{y=n+1-\epsilon} + \mathcal C (y-x,\beta) g(y)\Big\vert_{y=a+\delta-\epsilon}^{y=a+1-\epsilon} \notag \\ &-\mathcal C (y-x,\beta) g(y)\Big\vert_{y=b+\delta-\epsilon}^{y=b+1-\epsilon}\Bigg)-\lim_{\beta \searrow 0}\int\limits_{a+\delta}^{b+\delta} \mathcal C (y-x,\beta) g'(y)\, \mathrm d y ,
  \label{eq:int_proof1}
\end{align}
where the integrals have been combined to a single one by taking the limit $\epsilon \to 0$. Substracting \eqref{eq:int_proof1} from \eqref{eq:sum_proof1}, we obtain 
\begin{align}
&\sum_{n=a+1}^b f(n)- \int \limits_{a+\delta}^{b+\delta} f(y)\, \mathrm d y \notag \\  &= -\lim_{\beta \searrow 0}\lim_{\epsilon \searrow 0} \mathcal C (y-x,\beta)g(y)\Big\vert_{y=a+\delta-\epsilon}^{y=b+\delta-\epsilon}+\lim_{\beta \searrow 0}\int \limits_{a+\delta}^{b+\delta} \mathcal C (y-x, \beta) g'(y) \,\mathrm d y.
\label{eq:proof1_one_limit_remains}
\end{align}
As $\mathcal C(\cdot, \beta)$ is left continuous ($\lceil \cdot \rceil$ is left continuous, c.f. Definition~\ref{def:c_function}), the limit $\epsilon\searrow 0$ in \eqref{eq:proof1_one_limit_remains} yields \eqref{eq:start_of_proof1}.  
\end{proof}

Before continuing the transformation of the right hand side \eqref{eq:start_of_proof1} into a differential operator, more properties of the function $\mathcal C $ are needed. We start with an investigation of the derivatives with respect to the weighting $\beta$, which are subsequently put into connection with antiderivatives of $\mathcal C $ with respect to $y$.
Note that the  definition below is well-defined as by Lemma~\ref{lem:Cs_beta_properties} the function $\mathcal C (\cdot, \beta)$
decays exponentially for $\beta > 0$.

\begin{definition} \label{def:antiderivatives}
  For $s\in S$, we define the consecutive antiderivatives of $\mathcal C (\cdot,\beta)$,
\begin{align*}
 \mathcal C_0 (y,\beta)&= \mathcal C (y,\beta), \\
  \mathcal C_{\ell+1} (y,\beta)&=-\int \limits_y^\infty \mathcal C_\ell (z,\beta)\, \mathrm d z, \quad \ell \in \mathds N,
\end{align*}
for $y > 0$ and $\beta > 0$.
\end{definition}

The iterated antiderivatives of $\mathcal C$ can be expressed explicitly by derivatives with respect to the regularisation
parameter. This form is very useful for deriving bounds on $(\mathcal C_\ell)_{\ell \in \mathds N}$.
Moreover, we are able to extend their definition to $\beta = 0$.

\begin{lemma}
  \label{lem:cj_beta_limit_exists}
  Let $s\in S$ and $\ell \in \mathds N$. The function $\mathcal C_\ell $ admits the explicit form 
  \begin{equation}
  \mathcal C_\ell (y,\beta)=\frac{1}{\ell!}\sum_{k=0}^\ell  \binom{\ell}{k} y^{\ell-k}  \partial_\beta^k \mathcal C (y,\beta),
  \quad y > 0, \quad \beta > 0,
  \label{eq:dfunctionsexplicit}
  \end{equation}
  which is is also valid in the limit $\beta \searrow 0$.
  In addition, $\mathcal C_\ell(\cdot, \beta) \in C^{\ell - 1}(0,\infty)$ for all $\beta \geq 0$.
  Relation~\eqref{eq:dfunctionsexplicit} can be compactly written as
  \begin{equation}
  \mathcal C_\ell (y,\beta)=\frac{1}{\ell!}\Big(y+\partial_\beta\Big)^\ell \mathcal C (y,\beta), \quad y > 0,~\beta \geq 0.
  \label{eq:cj_differential}
  \end{equation}
\end{lemma}

Above lemma is the basis for bounds on $(\mathcal C_\ell)_{\ell \in \mathds N}$.
They are needed in the proof of Theorem~\ref{th:singular_euler_entire} to perform the limit $\ell \to \infty$.

\begin{lemma}\label{lem:cl-estimate-pointwise}
  Let $s\in S_\alpha$ for $\alpha \in \mathds R$ with constants $c,c_0,\gamma\ge 1$. Set $\ell_\alpha=\max\{0,\lceil \alpha \rceil \} + 1$. Then
  \begin{align}
    \Big|\mathcal C_\ell (y,\beta)\Big|\le c_s \left( (\ell_\alpha + 1+\ell)^{\ell_\alpha+1} \, \tau^{-\ell} e^\tau \Big(\lceil y \rceil^\alpha +\lceil y \rceil^{-1}\Big)  + \frac{1}{(\ell + 1)!} \max(y^\alpha,\lceil y \rceil^\alpha) \right),
    \label{eq:c_ell_bound}
  \end{align}
  for all $\ell \in \mathds N$, and $y,\beta>0$,
  with $c_s>0$ only depending on $s$ and
  \begin{equation}
    \tau=\frac{2\pi}{\gamma+1}.
      \label{eq:c_alpha_tau_definition}
  \end{equation}
  The estimate holds in particular in the limit $\beta\searrow 0$.
\end{lemma}

As a direct consequence of Lemmas~\ref{lem:cj_beta_limit_exists} and~\ref{lem:cl-estimate-pointwise}, we get
\begin{lemma}\label{lem:al-definition-estimate}
Let $s \in S_\alpha$, $\alpha \in \mathds R$, and $\ell \in \mathds N$. The function
\[
\tilde{\mathcal A}_\ell : (0, \infty) \to \mathds C:~y \mapsto \ell! \lim\limits_{\beta \searrow 0} \mathcal C_\ell(y, \beta)
\]
lies in in $C^{\ell - 1}(0, \infty)$ and is estimated by
\[
\left| \frac{1}{\ell!} \tilde{\mathcal A}_\ell(y) \right| \le c_s \left( (\ell_\alpha +1+ \ell)^{\ell_\alpha+1} \, \tau^{-\ell} e^\tau\Big(\lceil y \rceil^\alpha +\lceil y \rceil^{-1}\Big)  + \frac{1}{(\ell + 1)!} \max(y^\alpha,\lceil y \rceil^\alpha) \right),
\]
for $y > 0$. Here, the constants are the same as in Lemma~\ref{lem:cl-estimate-pointwise}.
\end{lemma}
We now show that $(\tilde{\mathcal A}_\ell)_{\ell \in \mathds N}$ are the Bernoulli-$\mathcal A$ functions from Definition~\ref{def:bernoulliA}.
\begin{lemma}\label{lem:bernoulliA-generating}
We have
\[
\tilde{\mathcal A}_\ell = \mathcal A_\ell, \quad \ell \in \mathds N.
\]
\end{lemma}

\begin{proof}
  We show that $(\tilde{\mathcal A}_\ell)_{\ell \in \mathds N}$ and $(\mathcal A_\ell)_{\ell \in \mathds N}$
  have the same generating function.
  For $\xi > 0$ and $\ell \in \mathds N$, we compute
  \begin{align*}
   \tilde{\mathcal A}_\ell (\xi) &= \lim_{\beta\searrow 0} (\xi+\partial_\beta)^\ell \mathcal C (\xi,\beta) \\
   &=\lim_{\beta\searrow 0}  \sum_{k =0}^\ell \binom{\ell}{k} \xi^{\ell-k} \partial_\beta^k  \mathcal C (\xi,\beta) \\
   &=\lim_{\beta\searrow 0}  e^{-\beta \xi} \partial_\beta^{\ell} \Big( e^{\beta \xi} \mathcal C (\xi,\beta) \Big) \\ 
   &= \lim_{\beta \searrow 0} \partial_\beta^{\ell}\Big( e^{\beta \xi} \mathcal C (\xi,\beta) \Big),
  \end{align*}
  which proves
  \[
  e^{\beta \xi} \mathcal C(\xi, \beta) = \sum\limits_{\ell = 0}^\infty \tilde{\mathcal A}_\ell(\xi) \frac{\beta^\ell}{\ell!}.
  \] 
\end{proof}

We proceed with the proofs of the two main theorems.

\noindent \textbf{Proof of Theorem \ref{th:singular_euler_finite}:}
Let $x,a, b \in \mathds Z$, $x \le a < b$ and $\delta \in (0,1]$.
Furthermore, we pick $f : [a + \delta, b + \delta] \to \mathds C$ which factors into
\begin{equation*}
  f(y)=s(y-x) g(y), \quad y \in [a+\delta, b + \delta],
\end{equation*}
with $s \in S$ and $g \in  C^{\ell+1}[a+\delta,b+\delta]$.
We first show that
\[
\begin{aligned}
\sum_{n=a+1}^b f(n)-\int\limits_{a+\delta}^{b+\delta} f(y)\,\mathrm d y &= -\lim_{\beta \searrow 0} \Bigg(\sum_{k=0}^\ell (-1)^k \mathcal C_k (y-x,\beta)g^{(k)}(y) \Big\vert^{y=b+\delta}_{y=a+\delta} \\
&+\int \limits_{a+\delta}^{b+\delta}(-1)^\ell \mathcal C_ {\ell} (y-x,\beta) g^{(\ell+1)}(y) \,\mathrm d y\Bigg),
\end{aligned}
\]
for all $\ell \in \mathds N$.
For $\ell = 0$, this is proved in Proposition~\ref{prop:first_integral}.
The case $\ell \ge 1$ readily follows via iterated integration by parts, where successive
antiderivatives of $\mathcal C$ are given by $(\mathcal C_k)_{k \in \mathds N}$.
In order to take the limit $\beta \searrow 0$, we note that $\mathcal C_k(\cdot, \beta)$
is uniformly bounded in $\beta > 0$ for all $k \in \mathds N$ by Lemma~\ref{lem:cl-estimate-pointwise}.
Since $g^{(\ell + 1)}$ is continuous and therefore bounded on $[a+\delta, b+\delta]$,
the integrand is uniformly bounded in $\beta$ and we conclude by the dominated convergence theorem that
\[
\begin{aligned}
\sum_{n=a+1}^b f(n)-\int\limits_{a+\delta}^{b+\delta} f(y)\,\mathrm d y
&=
-\sum_{k=0}^\ell \frac{(-1)^k}{k!} \mathcal A_k (y-x)g^{(k)}(y) \Big\vert^{y=b+\delta}_{y=a+\delta} \\
&+ \frac{(-1)^\ell}{\ell!}\int \limits_{a+\delta}^{b+\delta} \mathcal A_{\ell} (y-x) g^{(\ell+1)}(y) \,\mathrm d y,
\end{aligned}
\]
where we have used Lemma~\ref{lem:bernoulliA-generating},
\[
\mathcal A_k(\xi) = k! \lim\limits_{\beta \searrow 0} \mathcal C_k(\xi, \beta), \quad \xi > 0.
\]

\noindent \textbf{Proof of Theorem \ref{th:singular_euler_entire}:} \\ 
  We now show that the limit $\ell\to \infty$ for \eqref{eq:new_great_euler_finite} in Theorem \ref{th:singular_euler_finite} exists, which implies Theorem \ref{th:singular_euler_entire}. First we analyse the behaviour of the remainder integral. We set 
  \begin{align}
    R_{\ell+1} = \frac{(-1)^\ell}{\ell!}\int \limits_{a+\delta}^{b+\delta}  \mathcal A_ {\ell} (y-x)  g^{(\ell+1)}(y)\, \mathrm d y.
  \end{align}
  Using Lemma \ref{lem:al-definition-estimate}, we find for all $k\in \mathds N$
  \begin{equation}
    \sup \limits_{y\in [a+\delta,b+\delta]} \Big|\frac{1}{k!}\mathcal A_k (y-x)\Big|=\mathcal O\Big(\tau^{-k}\Big),\quad k \to \infty,
  \end{equation}
  and furthermore by definition of $g$,
  \begin{equation}
    \sup \limits_{y\in [a+\delta,b+\delta]} \big| g^{(k)}(y)\big|=\mathcal O\Big((\sigma+\varepsilon)^{k}\Big),\quad k \to \infty,
  \end{equation}
  for all $\varepsilon>0$. This implies 
  \begin{equation}
    R_{\ell+1}= \mathcal O\Bigg(\left(\frac{\tau}{\sigma+\varepsilon}\right)^{-\ell}\Bigg),\quad \ell \to \infty.
  \end{equation}
As an immediate consequence of above estimates, the series  
\begin{equation}
  \sum_{k=0}^\infty \frac{(-1)^k}{k!} \mathcal A_k (y-x)g^{(k)}(y)
  \label{eq:absolutely_convergent_series}
\end{equation}
converges uniformly on $[a+\delta,b+\delta]$ for $\sigma<\tau$ and thus the limit $\ell \to \infty$ is well-defined.  

\noindent \textbf{Proof of Example \ref{ex:representation-zeta-function}: } \\
For $\beta > 0$, $y > 0$ and $s(y) = |y|^{-\nu}$ with $\nu \in \mathds C$,
\[
\begin{aligned}
\mathcal C (y, \beta) &= \sum\limits_{n= \lceil y \rceil}^\infty e^{-\beta n} n^{-\nu}  - \int\limits_{y}^\infty e^{-\beta z} z^{-\nu} \, \mathrm d z \\
&= e^{-\beta \lceil y \rceil} \sum\limits_{n=0}^\infty \frac{1}{(n + \lceil y \rceil)^\nu} e^{-\beta n} - \beta^{\nu - 1} \Gamma(1-\nu, \beta y),
\end{aligned}
\]
where $\Gamma(\cdot, \cdot)$ denotes the incomplete gamma function
\[
\Gamma(q, z) = \int\limits_{z}^\infty t^{q - 1} e^{-t} \, \mathrm d t, \quad z > 0,~q>0,
\]
and is subsequently continued to a meromorphic function on $\mathds C \times \mathds C$~\cite[Chap. IX]{Batemann1953b}.
For $\nu \in \mathds R \setminus \mathds N$, we have the following expansion for the series,
\[
e^{-\beta \lceil y \rceil} \sum\limits_{n=0}^\infty \frac{1}{(n + \lceil y \rceil)^\nu} e^{-\beta n}
=
\Gamma(1 - \nu, 0) \beta^{\nu - 1} + \sum\limits_{n=0}^\infty \zeta(\nu - n, \lceil y \rceil) \frac{(-1)^n}{n!} \beta^n,
\]
which is valid for all $\beta \in (0, 2\pi)$~\cite[Sec. 1.11, Eq. (8)]{Batemann1953a}.
The incomplete gamma function admits the power series, valid for all $\beta > 0$,~\cite[Sec. 9.2, Eq. (5)]{Batemann1953a},
\[
\beta^{\nu - 1} \Gamma(1 - \nu, \beta y) = \beta^{\nu - 1} \Gamma(1 - \nu, 0) - y^{-(\nu - 1)} \sum\limits_{n=0}^\infty \frac{y^n}{n+1-\nu} \frac{(-1)^n}{n!} \beta^n.
\]
Substracting both terms, the singularities cancel and we get
\[
\mathcal C (y, \beta) = \sum\limits_{n=0}^\infty (-1)^n \left[ \zeta(\nu - n, \lceil y \rceil) - \frac{y^{-(\nu - n -1)}}{\nu - n - 1} \right] \frac{\beta^n}{n!}.
\]
By above analysis, the radius of convergence of the series is $2 \pi$.
To obtain the generating function of the coefficients $(\mathcal A_\ell )_{\ell \in \mathds N}$, we compute  $e^{\beta y} \mathcal C (y, \beta)$
by means of the Cauchy product,
\[
e^{\beta y} \mathcal C(y, \beta) =
\sum\limits_{\ell = 0}^\infty \left[ \sum_{k=0}^\ell (-1)^k \binom{\ell}{k} y^{\ell - k} \left( \zeta(\nu - k, \lceil y \rceil) - \frac{y^{-(\nu - k -1)}}{\nu - k - 1} \right) \right] \frac{\beta^\ell}{\ell!}.
\]
This proves the form of the coefficients for $\nu \in \mathds R \setminus \mathds N$.
For integral $\nu$, above expression has an removable singularity.
Given $k \in \mathds N$, we write
\[
\zeta(\nu - k, \lceil y \rceil) - \frac{y^{-(\nu - k -1)}}{\nu - k - 1}
=
\zeta(\nu - k, \lceil y \rceil) - \frac{1}{\nu - k - 1} - \frac{y^{-(\nu - k -1)} - 1}{\nu - k - 1}.
\]
By Eq. (9) from \cite[Sec. 1.10]{Batemann1953a}, the first difference tends to
\[
\lim\limits_{\nu \to k + 1} \left( \zeta(\nu - k, \lceil y \rceil) - \frac{1}{\nu - k - 1} \right) = - \psi(y),
\]
where $\psi$ is the digamma function~\cite[Sec. 1.7]{Batemann1953a}.
The last term is the differential quotient of the function
\[
\nu \mapsto y^{-{\nu - k - 1}},
\]
evaluated at $k + 1$. Therefore, the limit is equal to $-\log y$.
In total, we have
\[
\lim\limits_{\nu \to k + 1} \left( \zeta(\nu - k, \lceil y \rceil) - \frac{y^{-(\nu - k - 1)}}{\nu - k - 1} \right)
=
-\psi(\lceil y \rceil) - \log y.
\]
The last term can be expressed as~\cite[Sec. 1.7.1, Eq. (9)]{Batemann1953a},
\[
-\psi(\lceil y \rceil) - \log y
=
\gamma_e - H_{\lceil y \rceil - 1} - \log y,
\]
where $\gamma_e$ is the Euler--Mascheroni constant and $H_k$ is the $k$th harmonic number,
\[
H_k = \sum\limits_{j=1}^k \frac{1}{j}, \quad k \in \mathds N.
\] 

\section{Technical Lemmas} \label{sec:technical_lemmata}

Before we prove the remaining lemmas from Section~\ref{sec:main_derivation},
we give the proofs for the two remarks in Section~\ref{sec:main_results}.
\begin{proof}[Remark~\ref{rem:gamma-algebraic-singularity}]
Let $\nu \in \mathds R$ and
\[
s: \mathds R^{*} \to \mathds R, ~y \mapsto |y|^{-\nu}.
\]
For $\ell \in \mathds N$, the $\ell$th derivative of $s$ is
\[
s^{(\ell)}(y) = (-1)^{\ell} \prod_{k=1}^{\ell} (k - 1 + \nu) \, \frac{1}{y^{\ell}} \, |y|^{-\nu}, \quad y \neq 0.
\]
The product in above formula has the alternative form
\[
\prod_{k=1}^{\ell} (k - 1 + \nu) = \frac{\Gamma(\nu + \ell)}{\Gamma(\nu)},
\]
where $\Gamma$ denotes the gamma function~\cite[Chap. I]{Batemann1953a}.
From Stirling's asymptotic expansion~\cite[Sec. 1.18]{Batemann1953a}, we know
\[
\Gamma(z) \sim \sqrt{\frac{2\pi}{z}} z^z e^{-z}, \quad z \to \infty,
\]
where $\sim$ means that the quotient of both sides tends to $1$ for $z \to \infty$.
Applying this in case of $\ell \to \infty$, we get
\[
\begin{aligned}
\frac{\Gamma(\nu + \ell)}{\Gamma(\nu) \ell!}
&=
\frac{\Gamma(\nu + \ell)}{\Gamma(\nu) \, \ell \Gamma(\ell)} \\
&\sim
\frac{1}{\ell \, \Gamma(\nu)} \frac{\ell}{\sqrt{\ell (\ell + \nu)}} e^{-\nu} \left( 1 + \frac{\nu}{\ell} \right)^\ell
\left( 1 + \frac{\nu}{\ell} \right)^{\nu} \ell^{\nu} \\
&\sim
\frac{1}{\Gamma(\nu)} \ell^{-(1-\nu)}, \quad \ell \to \infty.
\end{aligned}
\]
This shows that for $\nu \le 1$, the quotient of the factorial and the prefactor
is bounded.
In case of $\nu > 1$, it diverges algebraically which implies
\[
\frac{\Gamma(\nu + \ell)}{\Gamma(\nu) \ell !} \frac{1}{(1 + \varepsilon)^\ell} \to 0, \quad \ell \to \infty
\]
for all $\varepsilon > 0$.
To summarise, we have shown
\[
\begin{aligned}
\left| s^{(\ell)}(y) \right|
&= \frac{|\Gamma(\nu + \ell)|}{|\Gamma(\nu)|} |y|^{-\ell} |s(y)|\\
&\leq c \, \ell! \, (1 + \varepsilon)^\ell |y|^{-\ell} |s(y)|, \quad y \in \mathds R^*,
\end{aligned}
\]
with $c >0$ only depending on $\nu$ and $\varepsilon > 0$.
If $\nu \le 1$, the estimate also holds for $\varepsilon = 0$.  
\end{proof}
  
  \begin{proof}[Remark~\ref{rem:s_alpha-gronwall}]
  For $s \in S$, there are $c > 0$, $\gamma \ge 1$ such that
  \[
  |s'(y)| \le c \, \gamma \, |y|^{-1} |s(y)|, \quad y \in \mathds R^*.
  \]
  For $y > 1$, we have
  \[
  |s(y)| \le |s(1)| + \int\limits_{1}^y |s'(z)| \, \mathrm d z
  \leq |s(1)| + \int\limits_1^y c \gamma z^{-1} |s(z)| \, \mathrm d z,
  \]
  so by Gr\"onwall's inequality~\cite[Cor. 6.6]{Hale2009},
  \[
  |s(y)| \le |s(1)| \exp\left( c \gamma \int\limits_1^y z^{-1} \, \mathrm dz \right)
  = |s(1)| y^{\alpha},
  \]
  with $\alpha = c \gamma$.
  For $y < 0$, we set $\tilde y = - y$ and thereby extend the estimate to $\mathds R \setminus [-1,1]$.  
  \end{proof}

Albeit the definition of $\mathcal C$ is a priori not well-defined for $\beta = 0$, we can often
prove $\beta$ independent bounds and thereby extend the results in the limit $\beta \searrow 0$.
A key role plays the following simple estimate.
\begin{lemma}\label{lem:sbeta_estimate}
Let $s \in S$ with constants $c > 0, \gamma \geq 1$.
Then we have
\[
\left|s_\beta^{(\ell)}(y)\right| \leq c \, \ell! \, \gamma^\ell \, |y|^{-\ell} \, |s(y)|
\]
for all $y \in \mathds R^*$, $\ell \in \mathds N$ and $\beta > 0$.
\end{lemma}

\begin{proof}
For $\ell \in \mathds N$ and $y > 0$, we compute
\begin{align*}
|s_\beta^{(\ell)}(y)| &\leq \sum\limits_{k=0}^\ell \binom{\ell}{k} |s^{(k)}(y)| \beta^{\ell - k} e^{-\beta y} \\
&\leq c \, \ell! \, \gamma^\ell \, y^{-\ell} |s(y)| \left(\sum\limits_{k=0}^\ell \frac{1}{(\ell - k)! } y^{\ell - k} \beta^{\ell - k}\right) e^{-\beta y} \\
&\leq c \, \ell! \, \gamma^\ell \, y^{-\ell} |s(y)| e^{\beta y} e^{-\beta y} \\
&=c \, \ell! \, \gamma^\ell \, y^{-\ell} |s(y)|.
\end{align*}
Replacing $y$ by $-y$ in above estimate shows the result for all $y \in \mathds R^*$. 
\end{proof}

We now investigate the behaviour of $\mathcal C$ both as a function in $y$ and $\beta$.

\begin{proof}[Lemma~\ref{lem:Cs_beta_properties}]
Let $s \in S$ and $y > 0$.
To show that
\[
(0, \infty) \to \mathds C, ~\beta \mapsto \mathcal C (y, \beta)
=
\sum\limits_{n = \lceil y \rceil}^\infty s(n) e^{-\beta n} - \int\limits_y^\infty s(z) e^{-\beta z} \, \mathrm d z,
\]
is differentiable, it suffices to show that it is differentiable on every compact subinterval $[\beta_1, \beta_2]$
of $(0, \infty)$.
The integrand
\[
h: [y, \infty) \times [\beta_1, \beta_2] \to \mathds C,~(z,\beta) \mapsto s(z) e^{-\beta z}
\]
is a smooth function in both variables with partial derivative
\[
\partial_\beta h(z, \beta) = - z s(z) e^{-\beta z}, \quad \beta \in [\beta_1, \beta_2],~ z \geq y.
\]
Since $s$ is asymptotically smooth, it admits at most polynomial growth, see Remark~\ref{rem:s_alpha-gronwall}.
Thus, there is $c > 0$ with
\[
|z s(z)| e^{-\beta_1 z} \leq \frac{c}{1+ z^2}, \quad z \geq y.
\]
This means
\[
|\partial_\beta h(z, \beta)| \leq |z s(z)| e^{-\beta_1} \leq \frac{c}{1+z^2}, \quad \beta \in [\beta_1, \beta_2],~ z \geq y.
\]
Because the majorant is independent of $\beta \in [\beta_1, \beta_2]$ and summable, respectively integrable on $[y, \infty)$,
we conclude that $\mathcal C (y, \cdot)$ is differentiable with derivative
\[
\partial_\beta \mathcal C (y, \beta) = 
- \sum\limits_{n = \lceil y \rceil}^\infty n s(n) e^{-\beta n} + \int\limits_y^\infty z s(z) e^{-\beta z} \, \mathrm d z,
\quad \beta > 0.
\]
With $s \in S$, also
\[
\tilde s_\ell(z) = (-1)^\ell z^\ell s(z), \quad z \neq 0,
\]
is asymptotically smooth for all $\ell \in \mathds N$.
Since $\tilde s_\ell e^{-\beta \cdot}$ coincides with $\partial_\beta^\ell h(\cdot, \beta)$
on $(0, \infty)$ for all $\beta > 0$, we conclude inductively that $\mathcal C $ is infinitely differentiable with respect to $\beta$.
Furthermore, it suffices to prove the rest of the lemma for $\ell = 0$.

Fix $\beta > 0$. Since $s \in S$, there are $\alpha \in \mathds R$ and $c_0 > 0$ with
\[
|s(z)| \leq c_0 z^\alpha, \quad z \geq 1,
\]
and $c_{\alpha, \beta} > 0$ such that
\[
e^{-\beta z} \leq \frac{c_{\alpha, \beta}}{(1+z)^{\alpha + 2}}, \quad z \geq 1.
\]
Combining above estimates, we have
\[
\begin{aligned}
\sum\limits_{n = \lceil y \rceil}^\infty |s(n)| e^{-\beta n}
&=
e^{-\beta \lceil y \rceil} \sum\limits_{n = 0}^\infty |s(n + \lceil y \rceil)| e^{-\beta n} \\
&\leq
c_0 c_{\alpha, \beta} e^{-\beta \lceil y \rceil}
\sum\limits_{n = 0}^\infty
\left( \frac{n + \lceil y \rceil}{n + 1} \right)^\alpha \frac{1}{(n+1)^2} \\
&\leq \frac{\pi^2}{6} c_0 c_{\alpha, \beta} e^{-\beta y}
\begin{cases}
1, & \alpha < 0, \\
(1 + \lceil y \rceil)^\alpha, & \alpha \geq 0,
\end{cases}
\end{aligned}
\]
for all $y > 0$.
The same estimates show the exponential decay of the integral term.

To prove the existence of the limit $\beta \searrow 0$, we apply the EM expansion up to order $\ell \in \mathds N$,
\begin{equation}
\mathcal C (y,\beta)=-\int \limits_y^{\lceil y\rceil} s_\beta(y)\,\mathrm d y
+
\sum_{k=0}^{\ell} \frac{(-1)^k}{k!} \frac{B_{k+1}(1)}{k+1} s_\beta^{(k)}(\lceil y \rceil)+R_\ell(y,\beta).
\label{eq:d_euler_maclaurin}
\end{equation}
The remainder reads
\begin{equation*}
  R_\ell(y,\beta)=\frac{(-1)^{\ell}}{\ell!} \int \limits_{\lceil y \rceil}^\infty \frac{B_ {\ell+1}(1+z-\lceil z \rceil)}{\ell+1} s^{(\ell+1)}_\beta(z)\,\mathrm d z.
\end{equation*}
For $\beta \searrow 0$, the sum of first two terms in~\eqref{eq:d_euler_maclaurin} converges to
\[
-\int \limits_y^{\lceil y\rceil} s(y)\,\mathrm d y
+
\sum_{k=0}^{\ell} \frac{(-1)^k}{k!} \frac{B_{k+1}(1)}{k+1} s^{(k)}(\lceil y \rceil),
\]
so we only have to investigate $R_\ell$.

From Remark~\ref{rem:s_alpha-gronwall} we know that there is $\alpha \in \mathds R$ with $s \in S_\alpha$.
Combining this with estimate~\eqref{eq:function_class}, we see that $s^{(\ell+1)}$ is integrable for all $\ell \in \mathds N$
larger than or equal to $\ell_\alpha = \max\{0, \lceil \alpha \rceil\} + 1$.

With Lemma~\ref{lem:sbeta_estimate}, the modulus of the integrand in $R_{\ell_\alpha}$ is readily estimated by
\[
c \gamma^{\ell_\alpha} \frac{|B_{\ell_\alpha+1}(1 + z - \lceil z \rceil)|}{\ell_\alpha+1} z^{-2}
\]
for all $z \geq 1$.
This upper bound is integrable on $[1,\infty)$ as a product of a bounded and an integrable function.
We can therefore employ the dominated convergence theorem which yields
\begin{equation}\label{eq:d_euler_maclaurin_limit}
\begin{aligned}
\lim_{\beta \searrow 0} \mathcal C (y,\beta) =&-\int \limits_y^{\lceil y\rceil} s(z) \,\mathrm d z
+
\sum_{k=0}^{\ell_\alpha}
\frac{(-1)^k}{k!}
\frac{B_{k+1}(1)}{k+1} s^{(k)}({\lceil y \rceil})\\
&+\frac{(-1)^{\ell_\alpha}}{\ell_\alpha!} \int \limits_{\lceil y \rceil}^\infty
\frac{B_{\ell_\alpha+1}(1+z-\lceil z \rceil)}{\ell_\alpha +1 } s^{(\ell_\alpha+1)}(z)\,\mathrm d z.
\end{aligned}
\end{equation} 
\end{proof}

\begin{proof}[Lemma~\ref{lem:Cs_y_properties}]
Let $s \in S$ and $\ell \in \mathds N$, $\beta > 0$.
From Lemma~\ref{lem:Cs_beta_properties}, we know that
\[
\partial_\beta^\ell \mathcal C (y, \beta) = (-1)^\ell
\left( 
\sum\limits_{n = \lceil y \rceil}^\infty n^\ell s(n) e^{-\beta n}
-
\int\limits_{y}^\infty z^\ell s(z) e^{-\beta z} \, \mathrm dz  
\right), \quad y > 0.
\]
Using Eq.~\ref{eq:d_euler_maclaurin}, we see
\begin{equation}\label{eq:cs-derivative-em}
\begin{aligned}
\partial_\beta^\ell \mathcal C (y, \beta)
=&
-\int\limits_{y}^{\lceil y \rceil} (-1)^\ell z^\ell s_\beta(z) \, \mathrm d z
+
\sum\limits_{k=0}^{\ell_\alpha} \frac{(-1)^k}{k!}
\frac{B_{k+1}(1)}{k+1} (-1)^\ell \lceil y \rceil^\ell s_\beta^{(k)}(\lceil y \rceil) \\
&+\frac{(-1)^{\ell_\alpha}}{\ell_\alpha!}
\int\limits_{\lceil y \rceil}^\infty \frac{B_{\ell_\alpha+1}(1+z - \lceil z \rceil)}{\ell_\alpha + 1}
(-1)^\ell z^\ell s_\beta^{(\ell_\alpha + 1)}(z) \, \mathrm d z, \quad y > 0,
\end{aligned}
\end{equation}
where $\alpha \in \mathds R$ such that $s \in S_\alpha$ and $\ell_\alpha = \max\{0, \lceil \alpha \rceil\} + \ell +  1$
as in the proof of Lemma~\ref{lem:Cs_beta_properties}.
As shown in~\ref{eq:d_euler_maclaurin_limit}, above equation remains valid in the limit $\beta \searrow 0$.
On $(n, n+1), n \in \mathds N$, Eq.\ref{eq:cs-derivative-em} shows that $\partial_\beta^\ell \mathcal C$
is an antiderivative of a smooth function,
\[
\partial_\beta^\ell \mathcal C(y, \beta) = c_n - \int\limits_{y}^{n+1} (-1)^\ell z^\ell s_\beta(z) \, \mathrm d z,
\quad y \in (n, n+1),
\]
where $c_n$ is a constant only depending on $s$ and $n$.
This shows that $\partial_\beta^\ell \mathcal C$ is a smooth function in $y$ on $\mathds R_+ \setminus \mathds N$.

To prove the jump relation, fix $n \in \mathds N_+$. For $\varepsilon_1, \varepsilon_2 \in (0,1)$, it holds
\[
\begin{gathered}
\partial_\beta^\ell \mathcal C (n - \varepsilon_1, \beta) - \partial_\beta^\ell \mathcal C (n + \varepsilon_2, \beta) = \\
- \int\limits_{n - \varepsilon_1}^{n + \varepsilon_2} (-1)^\ell z^\ell s_\beta(z) \, \mathrm d z
+ \int\limits_{n}^{n+1} (-1)^\ell z^\ell s_\beta(z) \, \mathrm d z \\
+ \sum\limits_{k=0}^{\ell_\alpha} \frac{(-1)^k}{k!} \frac{B_{k+1}(1)}{k + 1} (-1)^\ell \lceil z \rceil^\ell s_\beta^{(k)}(\lceil z \rceil) \bigg\vert^{n - \varepsilon_1}_{z = n + \varepsilon_2} \\
+ \frac{(-1)^{\ell_\alpha}}{\ell_\alpha!} \int\limits_n^{n+1} \frac{B_{\ell_\alpha+1}(1 + z - \lceil z \rceil)}{\ell_\alpha + 1} (-1)^\ell z^\ell s_\beta^{(\ell_\alpha+1)}(z) \, \mathrm d z.
\end{gathered}
\]
Taking the limit $\varepsilon_1, \varepsilon_2 \searrow 0$ yields for the right hand side
\[
\begin{aligned}
\int\limits_{n}^{n+1} (-1)^\ell z^\ell s_\beta(z) \, \mathrm d z
&+ \sum\limits_{k=1}^{\ell_\alpha} \frac{(-1)^k}{k!} \frac{B_{k+1}(1)}{k + 1} (-1)^\ell z^\ell s_\beta^{(k)}(z) \bigg\vert^{n}_{z = n + 1} \\
&+ \frac{(-1)^{\ell_\alpha}}{\ell_\alpha!} \int\limits_n^{n+1} \frac{B_{\ell_\alpha+1}(1 + z - \lceil z \rceil)}{\ell_\alpha + 1} (-1)^\ell z^\ell s_\beta^{(\ell_\alpha+1)}(z) \, \mathrm d z \\
= (-1)^\ell n^\ell s_\beta(n).
\end{aligned}
\]
In the last step, we applied the EM expansion~\eqref{eq:Euler--Maclaurin-expansion}
with $a = n -1$, $b = n$ and $\delta = 1$ up to order $\ell_\alpha$.
 
\end{proof}

Above jump relations are the key for the proof of the alternative representation of $(\mathcal C_\ell)_{\ell \in \mathds N}$.

\begin{proof}[Lemma~\ref{lem:cj_beta_limit_exists}]
  Let $s \in S$. The base case $\ell = 0$ holds by definition.
  Let $\ell \in \mathds N$. We assume that the form~\eqref{eq:dfunctionsexplicit} holds for $\ell$,
  \begin{equation}\label{eq:cl-alternative-proof}
  \mathcal C_\ell(y, \beta) = \frac{1}{\ell !} \sum\limits_{k=0}^\ell \binom{\ell}{k} y^{\ell - k} \partial_{\beta}^k \mathcal C(y, \beta), \quad y > 0,~\beta > 0,
  \end{equation}
  and now prove that it holds for $\ell+1$.
  Once we established this relation, we can extend the definition of $\mathcal C_\ell$ to $\beta = 0$ since the right hand side
  is well-defined for $\beta = 0$, see Lemma~\ref{lem:Cs_beta_properties}.

  We recall the jump relation
  \begin{equation*}
    \lim_{y\nearrow n} \partial_\beta^k  \mathcal C (y,\beta) - \lim_{y\searrow n} \partial_\beta^k \mathcal C (y,\beta)= (-n)^k s_\beta(n)
  \end{equation*}
  for all $n \in \mathds N_+$, $k \in \mathds N$
  and the formula for mixed derivatives,
  \[
  \partial_y \partial_\beta^k \mathcal C(y, \beta) = (-y)^k s_\beta(y)
  \]
  for all $y \in \mathds R_+ \setminus \mathds N_+$.
  Note that these formulas hold for $\beta \ge 0$.

  The derivative of
  \[
  \tilde{\mathcal C}_{\ell + 1}(\cdot, \beta): \mathds R_+ \to \mathds C,~y \mapsto \frac{1}{(\ell + 1)!} \sum\limits_{k=0}^{\ell +1} \binom{\ell + 1}{k} y^{\ell + 1 - k} \partial_{\beta}^k \mathcal C(y, \beta), \quad \beta > 0
  \]
  at $y \in \mathds R_+ \setminus \mathds N_+$ reads
  \[
  \begin{aligned}
  &\frac{1}{(\ell + 1)!}  \sum\limits_{k=0}^{\ell} \binom{\ell + 1}{k} \Big( (\ell + 1 - k) y^{\ell - k} \partial_{\beta}^k \mathcal C(y, \beta) + y^{\ell + 1 - k} (-y)^k s_\beta(y) \Big)\\
  =& \frac{1}{\ell!} \sum\limits_{k=0}^{\ell} \binom{\ell}{k} y^{\ell - k} \partial_{\beta}^k \mathcal C(y, \beta)
  + \frac{1}{(\ell + 1)!} \sum\limits_{k=0}^{\ell + 1} \binom{\ell +1}{k} (-1)^k y^{\ell + 1} s_\beta(y) \\
  =& \mathcal C_\ell(y, \beta) + \frac{y^{\ell +1 } s_\beta(y)}{(\ell + 1)!} (1 - 1)^{\ell + 1} = \mathcal C_\ell(y, \beta).
  \end{aligned}
  \]
  To show that $\tilde{\mathcal C}_{\ell + 1}$ and $\mathcal C_{\ell +1}$ differ at most by a constant, we need to show that
  both functions are continuous.
  Since $\mathcal C_\ell(\cdot, \beta)$ is $C^{\ell - 1}$, its antiderivative $\mathcal C_{\ell + 1}(\cdot, \beta)$ is $C^{\ell}$, and therefore continuous.
  We already know that $\tilde{\mathcal C}_{\ell +1}(\cdot, \beta)$ is smooth on $\mathds R_+ \setminus \mathds N_+$,
  so to show continuity, we have to investigate its behaviour at the positive integers.
  By the jump relation, we compute
  \[
  \begin{aligned}
  \lim\limits_{y \nearrow n} \tilde{\mathcal C}_{\ell + 1}(y, \beta) - \lim\limits_{y \searrow n} \tilde{\mathcal C}_{\ell + 1}(y, \beta)
  =& \frac{1}{(\ell +1)!} \sum_{k=0}^{\ell + 1} \binom{\ell + 1}{k} n^{\ell + 1 - k} (-n)^k s_\beta(n) \\
  =& \frac{n^{\ell + 1} s_\beta(n)}{(\ell + 1)!} \sum\limits_{k=0}^{\ell + 1} \binom{\ell + 1}{k} (-1)^k = 0
  \end{aligned}
  \]
  for all $n \in \mathds N_+$.
  Hence, $\tilde{\mathcal C}_{\ell +1}(\cdot, \beta)$ is continuous.
  Consequently, it differs from $\mathcal C_{\ell +1}(\cdot, \beta)$ only by constant which is zero because both functions
  tend to zero at infinity, see Lemma~\ref{lem:Cs_beta_properties}.
  The jump relations are also valid for $\beta = 0$, so 
  combined with the induction hypothesis
  \[
  \lim\limits_{\beta \searrow 0} \mathcal C_{\ell}(\cdot, \beta) \in C^{\ell - 1}(0,\infty),
  \]
  above argument shows
  \[
  \lim\limits_{\beta \searrow 0} \mathcal C_{\ell + 1}(\cdot, \beta) \in C^{\ell}(0,\infty).
  \] 
\end{proof}

\begin{notation}
  Let $s\in S$, $\ell \in \mathds N$, $\xi\in \mathds R$, and $y\in \mathds R^*$. We then write
    \begin{equation}
      s_{\xi,\ell,\beta}(y)=\frac{1}{\ell!}(y-\xi)^\ell s_\beta(y).
    \end{equation}
\end{notation}

\begin{lemma}\label{lem:cell_explicit}
Let $s \in S$ and $\ell \in \mathds N$. The function $\mathcal C_\ell (\cdot,\cdot)$ takes the form
  \begin{equation}
    \mathcal C_\ell (y,\beta)=\sum_{n=\lceil y \rceil}^\infty s_{y,\ell,\beta}(n) - \int\limits_y^\infty s_{y,\ell,\beta}(z) \,\mathrm d z,
    \quad y > 0, \quad \beta > 0.
    \label{eq:cj_explicit}
  \end{equation}
\end{lemma}
\begin{proof}
Let $s \in S$.
From Lemma~\ref{lem:Cs_beta_properties}, we know
\[
\partial_\beta^k \mathcal C(y, \beta) = \sum\limits_{n = \lceil y \rceil}^\infty (-1)^k n^k s_\beta(n) - \int\limits_{y}^\infty (-1)^k z^k s_\beta(z) \, \mathrm d z
\]
for all $y > 0$, $\beta > 0$. Combining this with
\[
\mathcal C_\ell(y, \beta) = \frac{1}{\ell !} \sum\limits_{k = 0}^\ell \binom{\ell}{k} y^{\ell - k} \partial_\beta^k \mathcal C(y, \beta)
\]
from Lemma~\ref{lem:cj_beta_limit_exists}, the claim follows from direct application of the binomial theorem.  
\end{proof}

\begin{lemma}\label{lem:sbetaell_estimate}
  Let $s \in S$ with constants $c > 0, \gamma \geq 1$. For $y>0$, $\xi\ge 0$ and $y>\xi$, we have
  \begin{equation}
  \left|s_{\xi,\ell,\beta}^{(k)}(y)\right| \leq c \, k!/\ell! \, \gamma^{k-\ell} \big(\gamma(1-\xi/y)+1\big)^\ell\, |y|^{\ell-k} \, |s(y)|
  \label{eq:s3estimate_s_in_S}
  \end{equation}
  for $y \in \mathds R^*$, $k,\ell \in \mathds N$ with $k>\ell$, and $\beta > 0$. For $\xi=y$, we have 
  \begin{equation}
  \left|s_{y,\ell,\beta}^{(k)}(y)\right| \le \left\{\begin{matrix} c\,k!/\ell!\,\gamma^{k-\ell}\,|y|^{\ell-k}\,|s(y)|, &k\ge \ell,\\ 0, & k<\ell. \end{matrix} \right.
    \label{eq:s3estimate_xi_equals_y}
  \end{equation}
  \end{lemma}
\begin{proof}
    For $y>0$ we compute
    \begin{align*}
      s_{\xi,\ell,\beta}^{(k)}(y) &= \frac{1}{\ell!}\sum\limits_{j=0}^k j! \binom{k}{j} \binom{\ell}{j} (y-\xi)^{\ell-j} s_{\beta}^{(k-j)}(y). 
    \end{align*}
    In case of $\xi=y$, we obtain \eqref{eq:s3estimate_xi_equals_y} from Lemma \ref{lem:sbeta_estimate}. Otherwise
    \begin{align*}
      s_{\xi,\ell,\beta}^{(k)}(y)= k!/\ell!\,y^\ell  (1-\xi/y)^\ell \sum \limits_{j=0}^\ell  \binom{\ell}{j} (y-\xi)^{-j} \frac{s_\beta^{(k-j)}(y)}{(k-j)!}.
    \end{align*}
    Then by Definition \ref{def:asymptotically_smooth}, we obtain for $s\in S$
    \begin{align*}
      \Big|s_{\xi,\ell,\beta}^{(k)}(y)\Big|&\le c \, k!/\ell! \,\gamma^k  \,|y|^{\ell-k} \, |s(y)| \,(1-\xi/y)^{\ell} \sum \limits_{j=0}^\ell \binom{\ell}{j} (1-\xi/y)^{-j}\gamma^{-j} \notag \\
      &=c \, k!/\ell! \, \gamma^k \Big(1-\xi/y+\gamma^{-1}\Big)^\ell\, |y|^{\ell-k} \, |s(y)|.
    \end{align*} 
\end{proof}

\begin{proof}[Lemma~\ref{lem:cl-estimate-pointwise}]
Let $s\in S_\alpha$ with $\alpha\in \mathds R$ with constants $c > 0$ and $\gamma\ge 1$.
In the following, $c_s > 0$ denotes a generic constant that only depends on $s$ and may change between different equations.
For $\ell\in \mathds N$, $\xi>0$ and $\beta\ge 0$, the function 
\begin{equation*}
 \mathds R^*\to \mathds C,\quad y\mapsto s_{\xi,\ell,\beta}(y)=\frac{1}{\ell!}(y-\xi)^\ell s_{\beta}(y),
\end{equation*}
is asymptotically smooth and belongs to $S_{\alpha+\ell}$, which holds in particular for $\beta= 0$. We take the explicit form of the antiderivatives $\mathcal C_\ell(\cdot,\beta)$ from Lemma \ref{lem:cell_explicit},
\begin{equation*}
  \mathcal C_{\ell}(y,\beta)=\sum_{n=\lceil y \rceil}^\infty s_{y,\ell,\beta}(n)-\int\limits_{y}^\infty s_{y,\ell,\beta}(z)\,\mathrm d z.
\end{equation*}
We apply the EM expansion to the right hand side up to order $k_\alpha=\ell_\alpha+\ell$ with $\ell_\alpha = \max\{0, \lceil \alpha \rceil\} + 1$, which yields
  \begin{align*}
    \mathcal C_\ell(y,\beta) &=
    -\int \limits_y^{\lceil y\rceil}  s_{y,\ell,\beta}(z)\,\mathrm d z
    +\sum_{k=0}^{k_\alpha}\frac{(-1)^k}{k!} \frac{B_{k+1}(1)}{k+1}  s_{y,\ell,\beta}^{(k)}({\lceil y \rceil})\notag \\
    &+\frac{(-1)^{k_\alpha}}{k_\alpha!} \int \limits_{\lceil y \rceil}^\infty \frac{B_{k_\alpha+1}(1+z-\lceil z \rceil)}{k_\alpha+1}  s_{y,\ell,\beta}^{(k_\alpha+1)}(z)\,\mathrm d z.
  \end{align*}
  We now derive uniform bounds $\beta$ for $\mathcal C_\ell(y,\beta)$. We first consider $\lceil y\rceil \in \mathds N_+$ and then extend the result to $y\in \mathds R_+$ by means of a Taylor expansion. We have
  \begin{align}
    \left|\mathcal C_\ell (\lceil y \rceil,\beta) \right| &\le \sum_{k=0}^{k_\alpha}\left|\frac{B_{k+1}(1)}{(k+1)!} \right|  \left| s_{\lceil y\rceil,\ell,\beta}^{(k)}({\lceil y \rceil})\right| \notag \\ &+\int \limits_{\lceil y \rceil}^\infty \left| \frac{B_{k_\alpha+1}(1+z-\lceil z \rceil)}{(k_\alpha+1)!} \right| \left|  s_{\lceil y \rceil,\ell,\beta}^{(k_\alpha+1)}(z)\right|\,\mathrm d z.
    \label{eq:cj_absolute}
  \end{align}
  We take the estimate from Lemma \ref{lem:sbetaell_estimate},
  \begin{equation*}
  \left|s_{y,\ell,\beta}^{(k)}(y)\right| \le \left\{\begin{matrix} c\,k!/\ell!\,\gamma^{k-\ell}\,|y|^{\ell-k}\,|s(y)|, &k\ge \ell,\\ 0, & k<\ell, \end{matrix} \right.
  \end{equation*}
  and further use 
  \begin{equation*}
    \frac{k!}{\ell!}\le k^{k-\ell },\quad \ell\in \mathds N.
  \end{equation*}
  Moreover the Bernoulli polynomials obey, see Eq.~(19) and following discussion in~\cite{lehmer1940maxima},
  \begin{equation*}
   \max_{y\in [0,1]} \left|\frac{B_k(y )}{k!}\right|\le \frac{4}{(2\pi)^k},\quad k\in \mathds N.
  \end{equation*}
  From above estimates, we find that the first term on the right hand side of \eqref{eq:cj_absolute} is  bounded by 
  \begin{equation*}
    4 c\,(\ell_\alpha+\ell)^{\ell_\alpha} |s(\lceil y \rceil)|\sum_{k=\ell}^{\ell_\alpha+\ell } \frac{1}{(2\pi)^{k+1}} \left(\frac{\gamma}{\lceil y\rceil}\right)^{k-\ell}.
  \end{equation*}
  We obtain for the sum  
  \begin{equation}
    \sum_{k=\ell}^{\ell_\alpha+\ell } \frac{1}{(2\pi)^{k+1}} \left(\frac{\gamma}{\lceil y\rceil}\right)^{k-\ell}\le  (\ell_\alpha+1) \gamma^{\ell_\alpha} (2\pi)^{-\ell},
    \label{eq:estimate_c_ell_integer_sum}
  \end{equation}
  and therefore 
  \begin{equation*}
    \sum_{k=0}^{k_\alpha}\left|\frac{B_{k+1}(1)}{(k+1)!} \right|  \left| s_{\lceil y\rceil,\ell,\beta}^{(k)}({\lceil y \rceil})\right| \le c_s \lceil y \rceil^\alpha (\ell_\alpha+\ell)^{\ell_\alpha} (2\pi)^{-\ell}.
  \end{equation*}

We now analyse the remainder integral in \eqref{eq:cj_absolute}. We know from Lemma \ref{lem:sbetaell_estimate} that 
\begin{equation*}
  \left|s_{\xi,\ell,\beta}^{(k)}(z)\right| \leq c \, k!/\ell! \, \gamma^{k-\ell} \big(\gamma(1-\xi/z)+1\big)^\ell\, |z|^{\ell-k} \, |s(z)|,
  \end{equation*}
  where $z\ge \xi$ and  $k=\ell_\alpha+\ell+1$. After inserting $\ell_\alpha$ and by asymptotic smoothness of $s$, we find that the integrand is bounded by
\begin{equation*}
  c_s(\ell_\alpha+1+\ell)^{\ell_\alpha+1} \left(\frac{\gamma+1}{2\pi}\right)^\ell z^{-2},
\end{equation*}
and thus we obtain for the integral
\begin{equation*}
  \int \limits_{\lceil y \rceil}^\infty \left| \frac{B_{k_\alpha+1}(1+z-\lceil z \rceil)}{(k_\alpha+1)!} \right| \left|  s_{\lceil y \rceil,\ell,\beta}^{(k_\alpha+1)}(z)\right|\mathrm d z\le c_s(\ell_\alpha+1+\ell)^{\ell_\alpha+1}  \left(\frac{\gamma+1}{2\pi}\right)^\ell \lceil y \rceil^{-1}.
  \label{eq:estimate_c_ell_integer_remainder}
\end{equation*}
From above estimates follows
\begin{align}
\left| \mathcal C_\ell (\lceil y\rceil ,\beta) \right|\le c_s (\ell_\alpha+1+\ell)^{\ell_\alpha+1}\,\tau^{-\ell} \left(\lceil y \rceil^\alpha +\lceil y \rceil^{-1} \right),
  \label{eq:c_ell_bound_integer}
\end{align}
with  $\tau=2\pi/(\gamma+1)$.
For $y\in \mathds R_+\setminus \mathds N$, we proceed by expanding $\mathcal C_\ell (\cdot,\beta)$ around $\lceil y \rceil$. 
Using that for $\ell,k\in \mathds N$ and $k\le \ell$
\begin{equation*}
  \partial_y^k \mathcal C_\ell(y,\beta)= \mathcal C_{\ell-k}(y,\beta),
\end{equation*}
we find that
\begin{equation*}
  \mathcal C_\ell (y,\beta)=\sum_{k=0}^\ell \frac{1}{k!} \mathcal C_{\ell-k} (\lceil y \rceil,\beta) \,(y-\lceil y \rceil)^k + \frac{1}{(\ell+1)!} s_\beta(\xi)\,(y-\lceil y \rceil)^{\ell+1},
\end{equation*}
with $\xi \in (y,\lceil y \rceil)$. Using \eqref{eq:c_ell_bound_integer}, the absolute value is bounded by
\begin{equation*}
  |\mathcal C_\ell (y,\beta)| \le c_s (\lceil y \rceil^\alpha + \lceil y \rceil ^{-1}) \sum_{k=0}^\ell \frac{1}{k!} (\ell_\alpha+1+\ell-k)^{\ell_\alpha+1} \,\tau^{-(\ell - k)}+\frac{c_s \max(y^{\alpha},\lceil y \rceil^\alpha)}{(\ell+1)!}.
  \label{eq:c_ell_bound_part1}
\end{equation*}
We find for the sum 
\begin{align*}
  &\sum_{k=0}^\ell \frac{1}{k!} (\ell_\alpha+1+\ell-k)^{\ell_\alpha+1}\,\tau^{-(\ell - k)} \notag \\ &\le  (\ell_\alpha+1+\ell)^{\ell_\alpha+1} \tau^{-\ell} \sum_{k=0}^\ell \frac{1}{k!} \tau^{k} \notag \\ &\le (\ell_\alpha+1+\ell)^{\ell_\alpha+1} \tau^{-\ell} e^\tau.
  \label{eq:c_ell_bound_part2}
\end{align*}
Combining above estimates, it follows
\begin{align*}
  \Big|\mathcal C_\ell (y,\beta)\Big|\le c_s \left( (\ell_\alpha + 1+\ell)^{\ell_\alpha+1} \, \tau^{-\ell} \Big(\lceil y \rceil^\alpha +\lceil y \rceil^{-1}\Big)  + \frac{1}{(\ell + 1)!} \max(y^\alpha,\lceil y \rceil^\alpha) \right),
\end{align*}
uniformly in $\beta\ge 0$, which concludes the proof.  
\end{proof}

\section{Outlook}\label{sec:outlook}
We have tried to make this article as accessible as possible to a large audience without sacrificing mathematical rigour. It is our sincere hope that you, dear reader, will find the tools and results provided in this publication useful. We encourage you to develop the results further or use them to build something of interest. On the application side, the SEM expansion allows for a precise and fast evaluation of macroscopic lattice sums in solid state physics; here spin lattices come to mind. On the theory side, various extensions of the SEM are possible. Our next publication will be devoted to an extension of the SEM expansion to higher spatial dimensions. As a separate project, we aim at the development of fast and accurate algorithms for time evolutions of macroscopic crystals. Another point of interest is the determination of stationary states through the solution of the integro-differential equations that arise from the application of the SEM.

\section*{Acknowledgements}
We would first like to thank Daniel Seibel and Darya Apushkinskaya   for inspiring discussions. Our colleagues Peter Schuhmacher and Daniel Seibel  have taken the time to proof-read earlier version of this manuscript and have offered valuable suggestions which improved this work; we acknowledge your support. Finally, we would like to express our gratitude towards our supervisor Prof. Sergej Rjasanow, whose support and guidance made this work possible.


\end{document}